\documentclass[11pt,english,oneside,a4paper]{amsart}

\usepackage[english]{babel}
\usepackage[utf8]{inputenc}
\usepackage[T1]{fontenc}
\usepackage[bottom=3cm,a4paper]{geometry}
\usepackage[hidelinks]{hyperref}
\usepackage{amsmath,amsfonts,amssymb,mathtools,amsthm}
\usepackage{listings,verbatim,fancyvrb,xspace,enumerate}
\usepackage{xcolor,graphicx,ytableau,tikz}
\usepackage[all]{xy}

\newtheorem{theorem}{Theorem}[section]
\newtheorem*{theorem*}{Theorem}
\newtheorem*{mainthm}{Main Theorem}
\newtheorem{corollary}[theorem]{Corollary}
\newtheorem*{corollary*}{Corollary}
\newtheorem{lemma}[theorem]{Lemma}
\newtheorem{proposition}[theorem]{Proposition}
\theoremstyle{definition}
\newtheorem{example}[theorem]{Example}
\newtheorem*{example*}{Example}
\newtheorem*{examples*}{Examples}
\newtheorem{remark}[theorem]{Remark}
\newtheorem{definition}[theorem]{Definition}
\newsavebox{\eqbox}
\newenvironment{longequation*} {\begin{lrbox}{\eqbox}$} {$\end{lrbox}\begin{equation*}\resizebox{\linewidth}{!}{\ensuremath{\displaystyle\usebox{\eqbox}}}\end{equation*}}

\makeatletter
\newcommand*\bigcdot{\mathpalette\bigcdot@{.5}}
\newcommand*\bigcdot@[2]{\mathbin{\vcenter{\hbox{\scalebox{#2}{\(\m@th#1\bullet\)}}}}}
\newcommand{\tikznode}[3][inner sep=0pt]{\tikz[remember picture,baseline=(#2.base)]{\node(#2)[#1]{\(#3\)};}}
\makeatother

\def\bydef{\coloneqq}
\def\og{{\fontencoding{T1}\selectfont\char19}}
\def\fg{{\fontencoding{T1}\selectfont\char20}\xspace}

\DeclarePairedDelimiter{\abs}{\lvert}{\rvert}
\DeclarePairedDelimiter{\Set}{\lbrace}{\rbrace}
\newcommand{\suchthat}{\hskip.5ex\color{black!66}\middle/\color{black}\hskip1ex}

\DeclareMathOperator{\Span}{Span}
\DeclareMathOperator{\diff}{d}

\DeclareMathOperator{\Flag}{Flag}
\DeclareMathOperator{\Grass}{Grass}
\DeclareMathOperator{\rank}{rank}

\DeclareMathOperator{\codim}{codim}
\DeclareMathOperator{\Bs}{Bs}
\DeclareMathOperator{\Exc}{Exc}

\DeclareMathOperator{\pr}{pr}
\DeclareMathOperator{\Id}{Id}
\let\S\relax % ( was § )
\DeclareMathOperator{\S}{S}
\DeclareMathOperator{\Ext}{\mathchoice
  {\scalebox{1.2}{$\mathsf{\Lambda}$}}
  {\scalebox{1.1}{$\mathsf{\Lambda}$}}
  {\scalebox{0.8}{$\mathsf{\Lambda}$}}
  {\scalebox{.65}{$\mathsf{\Lambda}$}}
}

\newcommand{\C}{\mathbb{C}}
\renewcommand{\L}{\mathcal{L}} % ( was Ł )
\renewcommand{\O}{\mathcal{O}} % ( was Ø )
\renewcommand{\P}{\mathbf{P}} % ( was ¶ )
\newcommand{\R}{\mathbb{R}}
\newcommand{\N}{\mathbb{N}}
\newcommand{\B}{\mathbb{B}}
\newcommand{\kk}{\mathbf{k}}
\let\mathbi\boldsymbol
\renewcommand{\aa}{\mathbi{a}} % ( was å )

\newcommand{\GG}{\mathbi{G}}
\newcommand{\FF}{\mathbi{F}}
\newcommand{\X}{\mathcal{X}}

\author{Antoine Etesse}
\email{antoine.etesse@univ-amu.fr}
\address{Aix Marseille Univ, CNRS, Centrale Marseille, I2M, Marseille, France}
\title
[Ampleness of Schur powers of cotangent bundles and \(k\)-hyperbolicity]
{Ampleness of Schur powers of cotangent bundles\\and \(k\)-hyperbolicity}

\begin{document}

\begin{abstract}
In this paper, we study a variation of a conjecture of Debarre on positivity of
cotangent bundles of complete intersections. We establish the ampleness of
Schur powers of cotangent bundles of generic complete intersections in
projective manifolds, with high enough explicit codimension and
multi-degrees. Our approach is naturally formulated in terms of flag
bundles and allows one to reach the optimal codimension. On complex
manifolds, this ampleness property implies intermediate hyperbolic
properties. We give a natural application of our main result in this
context.
\end{abstract}

\maketitle

\section*{Introduction}
In~\cite{Deb2005}, Debarre proved that complete intersections of sufficiently ample general hypersurfaces in complex Abelian varieties whose codimensions are at least as large as their dimensions have ample cotangent bundles.
By analogy, he suggested that this result should also hold when Abelian varieties are replaced by complex projective spaces.
This has been recently proved by Xie~\cite{Xie2018} and Brotbek--Darondeau~\cite{BD2018}, independently, based on ideas and explicit methods developped in~\cite{Brotbek2016}.
It is well-known that complex compact manifolds with ample cotangent bundles are complex hyperbolic in the sense of Kobayashi (cf.~\cite{kob13}).

More generally, for \(k \geq 1\), the notion of \textsl{\(k\)-infinitesimal hyperbolicity} (See Sect.~\ref{se: hyperbolicity}) generalizes the usual notion of hyperbolicity. Indeed, Kobayashi hyperbolicity is in this situation equivalent to \(1\)-infinitesimal hyperbolicity, and \(1\)-infinitesimal hyperbolicity implies \(k\)-infinitesimal hyperbolicity, \(k\geq 1\). Thus, the ampleness of the cotangent bundle implies \(k\)-infinitesimal hyperbolicity. However, to obtain this weaker property, it is sufficient to assume the ampleness of the \(k\)th exterior power of the cotangent bundle (\cite[Prop. 3.4]{Santa}).

\subsubsection*{Main new results}
Exterior powers are examples of \textsl{Schur powers}. These are obtained by applying to vector bundles certain \textsl{Schur functors} associated to partitions 
\[
  \lambda
  =
  (\lambda_{1}\geq\dotsb\geq\lambda_{k}>0)
\]
(see Sect.~\ref{se:definition} for more details).
In this work, we extend and we generalize the results of~\cite{BD2018} to such general Schur powers.
Namely, we establish the following ampleness result.
\begin{mainthm}
  Let \(M\) be a projective variety of dimension \(N\) over an algebraically closed field \(\kk\) of characteristic zero, equipped with a fixed very ample line bundle \(\O_{M}(1)\).
  Let \(c \leq N\) be a positive integer.

  For any partition \(\lambda=(\lambda_{1} \geq \dotsb \geq \lambda_{k} >0)\) with \(k\) parts, such that \((1+k)c \geq N\),
  the \(\lambda\)th Schur power of the cotangent bundle of a complete intersection
  \[
    X
    \bydef 
    H_{1}\cap \dotsb \cap H_{c}
  \]
  of \(c\) generic hypersurfaces \(H_{i}\in |\O_{M}(d_{i})|\) with respective degrees \(d_{i}\) such that
  \[
    d_{i}
    \geq
    \left(
      1+
      2c \,
      \frac{2\lambda_{1}+\lambda_{2}+ \dotsb + \lambda_{k}}{\gcd(\lambda_{1}, \dotsc, \lambda_{k})}
      \bigl(N+k(N-k)\big)^{(1+k)c}
    \right)^{2}
  \]
  is ample.
\end{mainthm}
\begin{examples*}\mbox{}
  \begin{enumerate}
    \item For \(k=1\) (symmetric powers), one recovers Xie--Brotbek--Darondeau's theorem.

    \item For \(\lambda=(1,1)\) and \(N=5\), a complete intersection \(X\) of two generic hypersurfaces of degrees \(\geq 6.10^{16}\) has \(\Ext^{2}\Omega_{X}\) ample, whereas it cannot have \(\Omega_{X}\) ample (since \(2c=4 < 5\)).
  \end{enumerate}
\end{examples*}

Qualitatively, this statement implies that the \(\lambda\)th Schur power of the cotangent bundle of a complete intersection of at least \(N/(1+k)\) generic hypersurfaces with large enough degrees will always be ample. This codimension condition is optimal. 

Indeed, let us recall the following vanishing theorem of Brückmann and Rackwitz for \(M=\P^{N}\). Recall that for a partition \(\lambda\), the \textsl{conjugate} partition \(\lambda^{*}\) is defined by \(\lambda_{j}^{*} \bydef \#\Set*{i\suchthat\lambda_{i} \geq j}\).
\begin{theorem*}[\cite{BR}]
  Let \(X\) be a smooth complete intersection of codimension \(c\) in \(\P^{N}\). Let \(\lambda\) be a partition such that \((\lambda_{1}^{*}+ \dotsb + \lambda_{c}^{*})<(N-c)\). Then one has the following vanishing result:
  \[
    H^{0}(X,\S^{\lambda}\Omega_{X})
    =
    \Set{0}.
  \]
\end{theorem*}
Now, assume \(\lambda\) satisfies \((1+k)c<N\). On the one hand, when \(m\geq c\), since \(k=\lambda_{1}^{*}\), the condition of the previous theorem for the partition \(m\lambda\) becomes
\[
  \underbrace{k+\dotsb+k}_{\times c}
  <
  N-c,
\]
which is satisfied.
Therefore \(H^{0}(X,\S^{m\lambda}\Omega_{X})=\Set{0}\) for \(m\geq c\). But on the other hand, the ampleness of \(\S^{\lambda}\Omega_{X}\) would imply that \(\S^{m\lambda}\Omega_{X}\) is globally generated for \(m\) large enough, by Proposition~\ref{prop:LN} below. Accordingly, \(\S^{\lambda}\Omega_{X}\) cannot be ample.

As annonced, the particular case of exterior powers has the following interesting consequence in hyperbolicity.
\begin{corollary*}
  Let \(M\) be a complex projective variety of dimension \(N\), equipped with a very ample line bundle \(\O_{M}(1)\).
  The complete intersection of  \(c\geq N/(1+k)\) generic hypersurfaces \(H_{i} \in |\O_{M}(d_{i})|\) of degree \(d_{i} \geq \big(1+2c(k+1)(N+k(N-k))^{c(k+1)+1}\big)^{2}\) is \(k\)-infinitesimally hyperbolic.
\end{corollary*}

This property has a geometric interpretation in terms of holomorphic maps \(f\colon\C \times \B^{k-1} \to X\): it implies that every such map must be everywhere degenerate, in the sense that the Jacobian matrix of \(f\) is nowhere of maximal rank (cf. Sect.~\ref{se: hyperbolicity}). For \(k=1\), one recovers as expected the notion of hyperbolicity in the sense of Brody.

\subsubsection*{Key arguments of the proof.}
We would like to avoid any redondancy with Sect.~\ref{se:overview}, where the proof of our main theorem is outlined in an introductive way. We will hence only briefly sketch its proof here, and we will try to underline the role of the original arguments introduced in the current work. 

Our proof goes naturally in the context of general Schur powers, which justifies in itself to work with such generality.
The central objects of this paper are flag bundles, and their positive line bundles.
A first key idea is to relate the positivity of Schur powers with the positivity of these line bundles, on associated flag bundles. This natural result, in the spirit of Hartshorne's point of view on ampleness of vector bundles, has been proved only very recently by Laytimi and Nahm (\cite{LN2018}). This work gives a natural first application of their result.
This is an important ingredient, since it allows to avoid plethysm phenomena, that would occur if one follows the most naive approach of projectivizing Schur powers.

Recall that the general scheme of the proofs in~\cite{BD2018} and in the associated papers on the positivity of complete intersections is to use explicit extrinsic equations for complete intersections in order to bring back the geometric situation under consideration to a model situation in which it is easy to obtain positivity.
The key input is to have sufficiently many natural equations to cut out these generically finite families.
More precisely, a partial evaluation of the defining equations and theirs differentials (or sometimes higher order jet differentials) allows to construct some morphisms from vector bundles over some well-chosen subfamilies of complete intersections into generically finite families over projective bases. The main feature of theses morphisms is that the pullback of the tautological line bundle gives a lot of negatively twisted sections of the bundle under consideration.
Besides, in the model situation, Nakamaye's theorem (or a variant of it) allows one to control the augmented base locus of the pullback of the tautological ample line bundle on the parameter space. 
Combining carefully this two facts, one is able to prove the sought ampleness result.

Our proof follows these very lines (see Sect.~\ref{se:overview}). In our case, the matter is to construct a morphism
\(\varPsi\) from (total spaces of) flag bundles of the relative tangent space of the universal family \(\X\to S\) of well-chosen particular complete intersections to  simpler generically finite families \(\mathcal{Y} \to \FF\) of subschemes in some projective spaces (see Diagram~\ref{eqref: diagram outline}).

In order to obtain sufficiently many equations, working with flag bundles is very important here again. It allows to work with scalar extrinsic equations instead of vectorial equations and therefore avoid substantial technical difficulties.

Of course, one has to be careful with the meaning of the derivative of an extrinsic equation.
In this work we slighlty improve the presentation of~\cite{BD2018} on this matter, by showing that the good point of view on these equations (or rather their partial evaluations) is to see them as maps in certain flag bundles over universal quotients. This allows to state in a more natural way the most technical results of~\cite{BD2018}, and to apply them in our more general context.

\subsubsection*{The paper is organized as follows.}\mbox{}

Section~\ref{se:definition} is devoted to briefly introducing Schur bundles and flag bundles;

Section~\ref{se:overview} gives a more detailed overview of the proof sketched above;

Section~\ref{se:proof} gives a complete and rigorous treatment of the details of the proof;

Section~\ref{se: hyperbolicity} contains applications to \(k\)-infinitesimal hyperbolicity.

\section*{Acknowledgements}
I would like to thank my supervisor Erwan Rousseau as well as my co-supervisor Lionel Darondeau for their help and support. This work owes a lot to Lionel's insights on the subject: I could never thank him enough for the time he spent sharing it with me.

\section{Ampleness, Schur bundles and flag bundles}
\label{se:definition}
\subsection{Ampleness}
\label{subsection: ampleness}
For the theory of ample line bundles, we refer to~\cite{Laz}.
We briefly recall here Hartshorne's point of view on ampleness of vector bundles.
Note however that we will rather deal with subspaces than with quotients in our exposition.

Let \(E\) be a vector bundle over a variety \(M\) and let \(\P(E^\vee)\) denote the projective bundle of lines in its dual \(E^\vee\).
The vector bundle \(E\) is said to be \textsl{ample} on \(M\) if the Serre line bundle \(\O_{\P(E^{\vee})}(1)\) is ample on \(\P(E^{\vee})\).

Let \(\pi\colon\P(E^{\vee})\to M\) be the natural projection. 
A particular case of Bott's formula~\cite{Bott} states that for \(m\geq 1\):
\begin{align}
  \label{Bott formula}
  \pi_{*}(\O_{\P(E^{\vee})}(m))
  =
  \S^{m}E,
\end{align}
where \(\S^{m}\) is the \(m\)th symmetric power, and that 
\[
  H^{0}\left(\P(E^\vee),\O_{\P(E^{\vee})}(m)\right)
  =
  H^{0}\left(M,\S^{m}E\right).
\]

Accordingly, the ampleness of the vector bundle \(E\) is equivalent to the ampleness of its symmetric powers \(\S^{m}E\), for \(m\geq 1\).
As we will see in the sequel, this point of view can be generalized to Schur powers, looking at appropriate associated flag bundles.

\subsection{Schur bundles and flag bundles}
\label{subsection: Schur bundles}
We introduce the terminology and some facts about partitions, flag varieties, flag bundles, and Schur bundles.

Any partition \(\lambda=(\lambda_{1}\geq\dotsb\geq\lambda_{k}>0)\) with \(k\) parts is associated to a \textsl{Young diagram of shape \(\lambda\)}, which is a collection of cells arranged in left-justified rows: the first row contains \(\lambda_{1}\) cells, the second \(\lambda_{2}\), and so on (cf. Figure~\ref{fig1} for an example).
\begin{figure}[h]
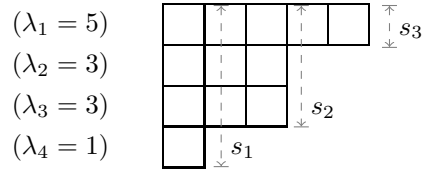

  \begin{center}
    \small
    \ytableausetup{centertableaux}
    \begin{ytableau}
      \none[(\lambda_{1}=5)]&\none[]&\none[]&&&&&&\none[\tikznode{s3}{~}]
      \\
      \none[(\lambda_{2}=3)]&\none[]&\none[]&&&
      \\
      \none[(\lambda_{3}=3)]&\none[]&\none[]&&&
      \\
      \none[(\lambda_{4}=1)]&\none[]&\none[]& & \none[]& \none[] & \none[]& \none[]& \none[\tikznode{s1}{~}]
    \end{ytableau}
    \tikz[overlay,remember picture]{%
      \draw[gray,|<->|,dashed] ([yshift=1em,xshift=-6.3em]s3.center) -- ([yshift=-.5em,xshift=-6.3em]s1.center) node[above right,black]{\(s_1\)};
      \draw[gray,|<->|,dashed] ([yshift=1em,xshift=-3.3em]s3.center) -- ([yshift=1em,xshift=-3.3em]s1.center) node[above right,black]{\(s_2\)};
      \draw[gray,|<->|,dashed] ([yshift=1em]s3.center) -- ([yshift=4em]s1.center) node[above right,black]{\(s_3\)};
    }
    \caption{Young diagram of the partition \(\lambda=(5,3^{2},1)\), with associated jump sequence \(s=(4,3,1)\).}
    \label{fig1}
  \end{center}
\end{figure}

By construction, if we read the Young diagram from top to bottom via the rows, we recover the partition \(\lambda\); observe that if instead we read it from left to right via the columns, we recover the conjugate partition \(\lambda^{*}\) mentionned in the introduction.
There is a \textsl{jump sequence} \(s=(s_{1}> \dotsb > s_{t})\) associated to a partition \(\lambda\), defined by:
\[
  \lambda_{i}>\lambda_{i+1}
  \iff
  (i\in s)
\]
(where by convention \(\lambda_{k+1}=0\)).
Equivalently, \(s_{1}>\dotsb>s_{t}\) are the (pairwise distinct) lengths of the parts of the conjugate partition \(\lambda^{\ast}\) (cf. Figure~\ref{fig1}).

Let \(V\) be a \(\mathbf{k}\)-vector space of dimension \(n\).
For a decreasing sequence of integers \[s=(s_{1}>s_{2}>\dotsb>s_{t}),\] with \(t \geq 1\), and \(s_{1} \leq n-1\), the \textsl{partial flag variety \(\Flag_{s}(V)\) of sequence \(s\)} is defined as
\[
  \Flag_{s}(V)
  \bydef
  \Set*{
    V \supsetneq F_{1}  \supsetneq \dotsb \supsetneq F_{t} \supsetneq \Set{0}
    \suchthat
    \text{\(F_{i}\) is a \(\kk\)-vector subspace of dimension \(s_{i}\)}
  }.
\]
On \(\Flag_{s}(V)\), there is a filtration of the trivial bundle \(\Flag_{s}(V)\times V\) by universal subbundles \(U_{s_{1}}\supsetneq\dotsb\supsetneq U_{s_{t}}\), given over the point \(\eta=(F_{1}\supsetneq \dotsb \supsetneq F_{t})\in\Flag_{s}V\) by \(U_{s_{i}}(\eta)=F_{i}\).

\begin{example}
  \label{ex: Grass}
  If we take the sequence \(s=(k)\), where \(k \geq 1\), then \(\Flag_{s}(V)\) is the Grassmannian of \(k\)-dimensional subspaces of \(V\), denoted \(\Grass(k, V)\), together with the tautological subbundle \(U_{k}\).
\end{example}

To each partition \(\lambda\) with jump sequence \(s\), one associates a line bundles \(\L_{\lambda}(V)\) on \(\Flag_{s}(V)\) by setting:
\[
  \L_{\lambda}(V)
  \bydef
  \bigotimes_{i=1}^{t}
  \det((U_{s_{i}}/U_{s_{i+1}})^{\vee})^{\lambda_{s_{i}}},
\]
where by convention \(U_{t+1}=\Set{0}\).
\begin{example}
  For the partition \(\lambda=(m^{k})\), \(\L_{\lambda}(V)\) is the \(m\)th tensor power of the Plücker line bundle \(\O(1)=\det(U_{k}^{\vee})\) on \(\Grass(k,V)\). This line bundle allows one to embed \(\Grass(k,V)\) in the projective space \(\P(\Ext^{k}V)\).
\end{example}

Let us now take \(E\) a vector bundle of rank \(n\) over a variety \(M\) and \(\lambda\) a partition with jump sequence \(s=(s_{1}>s_{2}>\dotsb>s_{t})\). The previous considerations make sense fiberwise, and the constructions globalize.
This allows one to define
\begin{itemize}
  \item{} the projective bundle \(\pi\colon\Flag_{s}(E^{\vee})\to M\), such that the fiber over \(x \in M\) is \(\Flag_{s}(E^{\vee}_{x})\);
  \item{} the line bundle \(\L_{\lambda}(E^{\vee})\) over \(\Flag_{s}(E^{\vee})\), such that \((\L_{\lambda}(E^{\vee}))_{| \Flag_{s}(E^{\vee}_{x})}= \L_{\lambda}(E^{\vee}_{x})\).
\end{itemize}
\begin{definition}
  \label{defi:schur}
  The \textsl{Schur bundle} associated to a partition \(\lambda\) with jump sequence \(s\) is the direct image
  \[
    \S^{\lambda}(E) 
    \bydef 
    \pi_{*}(\L_{\lambda}(E^{\vee})).
  \]
\end{definition}

This, together with the vanishing of \(R^{i}\pi_{*}(\L_{\lambda}E^{\vee})\) for \(i\geq 1\), is a proposition known as Bott's formulas in the literature (cf.~\cite{Bott}), but we take it as a definition.

Observe that for \(\lambda=(m)\) one recovers Formula~\eqref{Bott formula}. In other word, \(\S^{(m)}E=\S^{m}E\).
More generally, \(\S^{\lambda}E\) has a standard description as a quotient of products of symmetric powers, and we will now briefly sketch how to recover it from Definition~\ref{defi:schur}.
\begin{proposition}[{\cite[Sect. 2.1,Theo. 4.1.8]{Wey}}]
  \label{prop: Schur module}
  The Schur power \(\S^{\lambda}E\) is a quotient of
  \[
    \S^{\lambda_{s_1}}(\Ext^{s_{1}}E)
    \otimes
    \S^{(\lambda_{s_{2}}-\lambda_{s_{1}})}(\Ext^{s_{2}}E)
    \otimes
    \dotsb
    \otimes
    \S^{(\lambda_{s_{t}}-\lambda_{s_{t-1}})}(\Ext^{s_{t}}E).
  \]
\end{proposition}
\begin{proof}
  The partial flag bundle \(\Flag_{s}(E^{\vee})\) has a natural structure of projective bundle: the product of \textsl{Plücker embeddings}
  \[
    i
    \colon
    \Flag_{s}(E^{\vee})
    \longrightarrow
    \P(\Ext^{s_{1}}E^{\vee})\times\dotsb\times\P(\Ext^{s_{t}}E^{\vee})
  \]
  is fiberwise injective, and its image above a point \(x\in M\) is a closed algebraic subset of
  \[
    \P(\Ext^{s_{1}}E_{x}^{\vee})\times\dotsb\times\P(\Ext^{s_{t}}E_{x}^{\vee})
  \]
  (this product of projective varieties being itself projective via the \textsl{Segre embedding}).

  Observe that \(\L_{\lambda}(E^{\vee})\) can be rewritten:
  \[
    \L_{\lambda}(E^{\vee})
    =
    \det(U_{s_{1}}^{\vee})^{\lambda_{s_{1}}}
    \otimes
    \det(U_{s_{2}}^{\vee})^{\lambda_{s_{2}}-\lambda_{s_{1}}}
    \otimes\dotsb\otimes
    \det(U_{s_{t}}^{\vee})^{\lambda_{s_{t}}-\lambda_{s_{t-1}}},
  \]
  where for \(1\leq i \leq t\), \(U_{s_{i}}\) is the relative version of the tautological subbundles introduced above.
  Now, by the very definition of the embedding \(i\), this line bundle is actually nothing but the pullback bundle
  \[
    \L_{\lambda}(E^{\vee})
    =
    i^{*}
    \left(
    \O_{\P(\Ext^{s_{1}}E^{\vee})}(\lambda_{s_1})
    \boxtimes
    \O_{\P(\Ext^{s_{2}}E^{\vee})}(\lambda_{s_2}-\lambda_{s_1})
    \boxtimes
    \dotsb
    \boxtimes
    \O_{\P(\Ext^{s_{t}}E^{\vee})}(\lambda_{s_{t}}-\lambda_{s_{t-1}})
    \right),
  \]
  where \(\boxtimes\) denotes the \textsl{external tensor product}\footnote{Let \(X, Y\) be varieties, and let \(p_{1}\colon X\times Y \to X\), \(p_{2}\colon X \times Y \to Y\) be the first and second projection. If \(E_{1}\) is a vector bundle over \(X\), and \(E_{2}\) is a vector bundle over \(Y\), we denote \(E_{1}\boxtimes E_{2} \to X\times Y\) the vector bundle
    \[
      E_{1}
      \boxtimes E_{2}
      \bydef
      pr_{1}^{*}E_{1}
      \otimes
      pr_{2}^{*}E_{2}.
    \]
  This definition generalizes to an arbitrary finite product of varieties.}.
  Since \(i\) is a closed immersion, there is a natural surjective map
  \begin{equation}
    \label{eq:det}
    \O_{\P(\Ext^{s_{1}}E^{\vee})}(\lambda_{s_1})
    \boxtimes
    \O_{\P(\Ext^{s_{2}}E^{\vee})}(\lambda_{s_2}-\lambda_{s_1})
    \boxtimes
    \dotsb
    \boxtimes
    \O_{\P(\Ext^{s_{t}}E^{\vee})}(\lambda_{s_{t}}-\lambda_{s_{t-1}})
    \twoheadrightarrow
    i_{*}\L_{\lambda}(E^{\vee}).
  \end{equation}
  Let \(p\) be the natural projection from the total space of
  \[
    \O_{\P(\Ext^{s_{1}}E^{\vee})}(\lambda_{s_1})
    \boxtimes
    \O_{\P(\Ext^{s_{1}}E^{\vee})}(\lambda_{s_2}-\lambda_{s_1})
    \boxtimes
    \dotsb
    \boxtimes
    \O_{\P(\Ext^{s_{t}}E^{\vee})}(\lambda_{s_t}-\lambda_{s_t-1})
  \]
  to \(M\).
  Noticing that \(p_{*}i_{*}=\pi_{*}\),
  and applying \(p_{*}\) to \eqref{eq:det} yields a map
  \[
    \S^{\lambda_{s_1}}(\Ext^{s_{1}}E)
    \otimes
    \S^{(\lambda_{s_{2}}-\lambda_{s_{1}})}(\Ext^{s_{2}}E)
    \otimes
    \dotsb
    \otimes
    \S^{(\lambda_{s_{t}}-\lambda_{s_{t-1}})}(\Ext^{s_{t}}E)
    \to
    \S^{\lambda}E.
  \]
  This map is still surjective (see e.g. the proof of Bott's theorem in~\cite{Wey}),
  and its kernel constitutes the so-called \textsl{exchange relations}. One recovers the familiar description of the Schur power \(\S^{\lambda}E\) as a quotient of
  \[
    S^{\lambda_{s_1}}(\Ext^{s_{1}}E)
    \otimes
    S^{(\lambda_{s_{2}}-\lambda_{s_{1}})}(\Ext^{s_{2}}E)
    \otimes
    \dotsb
    \otimes
    S^{(\lambda_{s_{t}}-\lambda_{s_{t-1}})}(\Ext^{s_{t}}E).
    \qedhere
  \]
\end{proof}

\subsection{Ampleness of Schur bundles}
It is not very surprising that the point of view of Hartshorne on ampleness generalizes to Schur powers.
In~\cite{Demailly}, Demailly proved that if \(E\) is ample on \(M\), then \(\L_{\lambda}(E^{\vee})\) is ample on \(\Flag_{s}(E^{\vee})\). The statement can be improved by weakening the hypothesis on \(E\): it is actually enough to suppose that \(\S^{\lambda}E\) is ample in order to deduce the ampleness of \(\L_{\lambda}(E^{\vee})\).
A natural question is to ask wether the converse is true, namely, does the ampleness of the line bundle \(\L_{\lambda}(E^{\vee})\) imply the ampleness of the vector bundle \(\S^{\lambda}E\)? 
In~\cite{LN2018}, the authors answered the question:
\begin{proposition}[\cite{LN2018}]
  \label{prop:LN}
  Let \(\lambda\) be  a partition with jump sequence \(s\), and let \(E\) a vector bundle over \(M\). The Schur bundle
  \(\S^{\lambda}E\) is ample on \(M\) if and only if the line bundle \(\L_{\lambda}(E^{\vee})\) is ample on \(\Flag_{s} E^{\vee}\).
\end{proposition}

To our knowledge, the current work is the first application of this very nice result.

\begin{remark}
  \label{rema:LN}
  A direct consequence is that \(\S^{\lambda}E\) is ample if and only if \(\S^{m\lambda}E\) is ample. Note that this is a particular case of a result known as dominance theorem for quite a long time and that the proof in~\cite{LN2018} proceeds rather in the other direction.
\end{remark}

\begin{example}
  In the case \(\lambda=(1^{k})\), the proposition states that
  \(\Ext^{k} E\) is ample on \(M\)
  if and only if
  the Plücker line bundle is ample on \(\Grass(k, E^{\vee})\).
\end{example}
\begin{comment}
  As a corollary, one has the following:
  \begin{corollary}
    \label{cor LN}
    Let \(\lambda\) be  a partition with jump sequence \(s\), and \(E\) a vector bundle over \(M\). Then \(\S^{\lambda}E\) is ample if and only \(\S^{m\lambda}E\) is very ample for some \(m>0\). Accordingly, \(\S^{\lambda}E\) is ample if and only if \(S^{p\lambda}E\) is ample for any (some) \(p\geq 1\).
  \end{corollary}
  \begin{proof}
    The previous proposition allows us to work with the line bundle \(\L_{\lambda}(E^{\vee})\). Noticing that \((\L_{\lambda}(E^{\vee}))^{\otimes m}=\L_{m\lambda}(E^{\vee})\), the very definition of \(\S^{m\lambda}E^{\vee}\) as the push-forward by \(\pi\) of \(\L_{m\lambda}(E^{\vee})\) yields the result.
  \end{proof}
\end{comment}

\subsection{Embedded Schur powers}
\label{subsection: extrinsic}
The first obstacle one encounters in trying to generalize the method used in~\cite{BD2018} to the case of Schur powers is in writing extrinsic equations defining the Schur power of the cotangent bundle of a complete intersection inside the Schur power of the cotangent bundle of the ambiant space. 
Let us illustrate this with the example of exterior powers.
Consider a generic homogeneous polynomial \(E\in H^{0}(M,\O_{M}(d))\), and the hypersurface \(H\bydef(E=0)\), assumed to be smooth. Let \(V\) be a trivializing open set for \(O_{M}(1)\). 

The embedded tangent space \(TH\) in \(TM\) is given on \(V\) by:
\[
  TH_{|V}
  =
  \Set*{
    (x,v) \in TM_{|V}
    \suchthat
    \substack{
      E(x)=0
      \\
      \diff E(x,v)=0 
    }
  }.
\]
We see that the total space \(TH\) is of codimension \(2\) inside \(TM\), and that we have two natural scalar equations defining it.

Let us continue this illustration with the case of the higher exterior powers. The natural equation that, added to the equation \(E=0\), defines locally \(\Ext^{k}TH\) inside \(\Ext^{k}T\P^{N}\) becomes vectorial.  
It is the linear map given on pure tensors by
\[
  \diff E
  (x, v_{1} \wedge \dotsb \wedge v_{k}) 
  =
  \sum_{j=1}^{k} (-1)^{j}\diff E(x, v_{j})
  {v_{1}\wedge\dotsb\wedge v_{j-1}\wedge v_{j+1}\wedge\dotsb \wedge v_{k}}.
\]
Vectorial equations are not easily dealt with in our approach.
However, if one restricts oneself to pure tensors \((x,v_{1} \wedge ... \wedge v_{k})\), the vectorial equation splits into \(k\) scalar equations: \(\diff E(x, v_{1})=0,\dotsc,\diff E(x,v_{k})=0\). 
Since non-zero pure tensors parametrize the Grassmann bundle \(\Grass(k, TM)=\Flag_{(k)}TM\), we get that:
\[
  \Grass(k, TH)
  \overset{loc}{=}
  \Set*{
    (x, [v_{1}\wedge... \wedge v_{k}])
    \in \Grass(k, TM)
    \suchthat
    \substack{
      E(x)=0,\\
      \diff E(x, v_{1})
      =0\\
      \vdots\\
      \diff E(x, v_{k})
      =
      0
    }
  }.
\]
The bundle \(\Ext^{k}TH\) can then be thought of as the Grassmann bundle \(\O(1)\) on this variety.

This phenomenon generalizes to arbitrary Schur powers. Consider a partition \(\lambda\) with jump sequence \(s=(s_{1}>\dotsb>s_{t})\). In analogy to the case of exterior powers, one has:
\[
  \Flag_{s}(TH)
  \overset{loc}{=}
  \Set*{
    (x, F_{1} \supsetneq ... \supsetneq F_{t})
    \in \Flag_{s}(TM)
    \suchthat
    \substack{
      E(x)=0,\\
      \diff E(x,F_{1})=\Set{0}
    }
  }.
\]
Here, in order to get the good number of equations, one has to consider only pure tensors that are adapted to the flag.
The bundle \(\S^{\lambda}TH\) can then be thought of as the line bundle \(\L^{\lambda}(TH)\) on this variety, which is the restriction of \(\L^{\lambda}(TM)\).

\begin{remark}
  Note the following subtelty: here we could consider the differential of \(E\) on the pointed cone \(\hat{M}\) over \(M\) given by the very ample line bundle \(\O_{M}(1)\) and \(v_{i}\in T\hat{M}\). Indeed, \(TM\) can be seen as a quotient of \(T\hat{M}\) (this gives rise to the famous Euler sequence), and this quotient behaves well with respect to subvarieties. This is actually a natural choice regarding our strategy and from now on, we will only imply this convention. The advantage is that we can use homogeneous coordinates \(\xi_{0},\dotsc,\xi_{N}\) and their differentials, but the counterpart is that \(\diff E\) is well-defined only on a trivializing subset for \(O_{M}(1)\) (because \(\diff\xi\) is not a global section on \(M\)).
\end{remark}

To sum up, the advantage of working with the appropriate flag variety \(\Flag_{s}(TM)\) instead of the projectivization \(\P(\S^\lambda TM)\) is twofold:
we have seen that we can study ampleness of Schur powers via natural line bundles,  therefore avoiding the problems of plethysm (we study \(\S^{m\lambda}\) instead of \(\S^{m}\S^{\lambda}\)), and we get scalar equations instead of vectorial equations.

\section{Outline of the proof}
\label{se:overview}
In this section we give a short outline of the proof, containing all relevant arguments, but neglecting some unavoidable technical difficulties along coordinate hyperplanes.
In the full proof, this difficulties leads to actually work on the stratification of \(\Flag_{s}(TM)\) induced by these hyperplanes, with refined arguments.

\subsection{Setting and notation}
\label{sse:setting}
For the rest of the paper, we fix an algebraically closed field \(\kk\) of characteristic \(0\), a projective algebraic variety \(M\) of dimension \(N\), equipped with a fixed very ample line bundle \(\O_{M}(1)\), and \((N+1)\) sections \(\xi_{0}, \dotsc, \xi_{N}\) of \(\O_{M}(1)\) in general position. This means that each \(D_{i}\bydef \{\xi_{i}=0\}\) is smooth, and that the divisor \(D=\sum_{i} D_{i}\) is simple normal crossing. In the special case \(M=\P^{N}\), the sections \(\xi_{0}, \dotsc, \xi_{N}\) are homogeneous coordinates in \(\P^{N}\). 

Ampleness being an open property in families, in order to prove our main theorem, we are allowed to work with special equations.
We will work with hypersurfaces that can (slightly abusively) be called \textsl{Fermat-type hypersurfaces} (or more rigorously deformations of such Fermat-type hypersurfaces).
For \(\varepsilon \geq 1\), \(\delta \geq 1\), and \(r\geq 1\), these equations are modeled after the equation
\[
  \sum_{\abs{J}=\delta} a_{J}(x) \mathbi{\xi}(x)^{(r+1)J}
  \in
  H^{0}(M,\varepsilon+(r+1)\delta),
\]
where each \(a_{J}\) is a section of \(\O_{M}(\varepsilon)\),
and where we use the standard multi-index notation for sequences \(J = (j_{0},\dotsc,j_{N})\)
with \textsl{length} \(\abs{J} \bydef j_{0}+\dotsb+j_{N}=\delta\).
Taking \(\varepsilon=0\), one would therefore recover the genuine Fermat-type hypersurfaces.
We denote by 
\[
  N_{\delta} 
  \bydef 
  \dim H^{0}\left(\P^{N}, \O_{\P^{N}}(\delta)\right)
\]
the number of parameters defining a homogeneous polynomial in \((N+1)\) variables of degree \(\delta\in\N\).
Our equations are then parametrized by \(r\) and by
\[
  S_{\varepsilon,\delta}
  \bydef 
  H^{0}(M,\O_{M}(\varepsilon))^{N_{\delta}}.
\] 
In the sequel \(r\) is a global parameter that will be fixed later according to our needs.

Now, for collections \(\mathbi{\varepsilon}\), \(\mathbi{\delta}\) of \(c\) parameters, one can consider the parameter space
\[
  S_{\mathbi{\varepsilon},\mathbi{\delta}}
  \bydef
  S_{\varepsilon_{1},\delta_{1}}
  \times
  \dotsb
  \times
  S_{\varepsilon_{c},\delta_{c}}
\]
and the universal family \(\bar\X\) in \(M\times S_{\mathbi{\varepsilon},\mathbi{\delta}}\).
A general fiber of the family \(\bar{\X} \to S_{\mathbi{\varepsilon},\mathbi{\delta}}\) being smooth (see e.g.~\cite{BD2018}), there is an open dense subset
\(S \subset S_{\mathbi{\varepsilon},\mathbi{\delta}}\)
parametrizing smooth complete intersection varieties.
We denote \(\X\to S\) the corresponding universal family,
and denote by \(X_{\aa^{\bigcdot}}\bydef\pi_{\X}^{-1}(\aa^{\bigcdot})\) the member of \(\X\) parametrized by \(\aa^{\bigcdot}\in S\).

Next, we consider partitions \(\lambda=(\lambda_{1}\geq \dotsb \geq \lambda_{k}>0)\) with \(k\) parts and jump sequence \(s=(s_{1}>s_{2}>\dotsb>s_{t})\), with \(k=s_{1} \leq (N-1)\) and \(t\geq 1\).
Note the particularly important role of \(k\) in the codimension hypothesis of our main theorem.
As we are interested in ampleness of \(\lambda\)th Schur powers, Remark~\ref{rema:LN} allows us to assume without loss of generality that \(\gcd(\lambda_{1}, \dotsc, \lambda_{k})=1\). 
Consider then the partial flag bundle of the relative tangent bundle \(T_{\X/S}\) associated to the partition \(\lambda\)
\[
  \pi_{\X}
  \colon
  \Flag_{s}(T_{\X/S})
  \to
  \X,
\]
and the line bundle \(\L_{\lambda}(T_{\X/S})\to\Flag_{s}(T_{\X/S})\).
By Proposition~\ref{prop:LN} and by the openness property of ampleness, 
we need now to establish that 
\[
  \L_{\lambda}(T_{\X/S})_{\aa^{\bigcdot}}
  =
  \L_{\lambda}(T_{X_{\aa^{\bigcdot}}})
\]
is ample for some \(\aa^{\bigcdot}\in\S\).

Let \(E_{1},\dotsc,E_{c}\in S\) be the universal homogeneous equations cutting out \(\X\).
Let \(V\) be a trivializing open set for \(\O_{M}(1)\).
As we saw in Sect.~\ref{subsection: extrinsic}, a local description for \(\Flag_{s}(T_{\X/S}) \subset S \times \Flag_{s}(TM)\) is given by the universal equations \((E_{1},\dotsc, E_{c})\) and their relative differentials \((\diff E_{1},\dotsc,\diff E_{c})\):
\begin{longequation*}
  \Flag_{s}(T_{\X/S})
  \stackrel{loc}{=}
  \Set*{
    \left(
      \aa^{\bigcdot}=(\aa^{1}, \dotsc, \aa^{c}), (x,F_{1} \supsetneq \dotsb \supsetneq F_{t})
    \right)
    \in S\times \Flag_{s}(TM)
    \suchthat
    \substack{
      E_{1}(\aa^{1},x) = 0
      \\
      \vdots
      \\
      E_{c}(\aa^{c},x) = 0
      \\
      \diff E_{1}(\aa^{1},x,F_{1}) = \Set{0}
      \\
      \vdots
      \\
      \diff E_{c}(\aa^{c},x,F_{1}) = \Set{0}
    }
  }.
\end{longequation*}
Beware however that the (relative) differentials \(\diff E_{i}\) are taken on the pointed cone \(\hat M\) over \(M\) and makes sense only locally on \(M\).

A special feature of Fermat-type equations is that such equations and their differentials can formally be written in very similar forms. Indeed, one can write:
\[
  E_{p}(\aa^{p},x)
  =
  \sum_{\abs{J}=\delta_{p}} \alpha_{J}(\aa^{p},x)(\mathbi{\xi}(x)^{r})^{J},
\]
for homogeneous polynomials of degree \(\varepsilon_{p}+\delta_{p}\)
\[
  \alpha_{J}(\aa^{p},x)
  \bydef
  a_{J}^{p}(x) \mathbi{\xi}(x)^{J}
  \\
\]
and one can write in a similar way
\[
  \diff E_{p}(\aa^{p},x,v)
  \stackrel{\textit{loc}}{=}
  \sum_{\abs{J}=\delta_{p}} \theta_{J}(\aa^{p},x,v)(\mathbi{\xi}(x)^{r})^{J},
\]
where this time
\[
  \theta_{J}(\aa^{p},x,v)
  \bydef
  \mathbi{\xi}(x)^{J}\diff a_{J}^{p}(x,v) 
  +
  (r+1)a_{J}^{p}(x)\diff\mathbi{\xi}(x,v)^{J}.
\]
Here, the coefficients \(\theta_{J}(\aa^{p}, \cdot)\) are \(1\)-forms on \(\hat{M}\) depending on the choice of the trivialization for \(\O_{M}(1)_{| V}\).
Let us emphasize once again that this expression is local, and depends on the trivialization of \(\O_{M}(1)\), in the exact same way as the forms \(\diff\xi_{i}\) do.

However, we will now perform a partial evaluation of these equations and establish that, in this way, for each \(p\), the coefficients \(\alpha_{J}^{p}\) and \(\theta_{J}^{p}\) can be seen as the relative coordinates on a certain flag variety. Moreover, we will show that this natural construction does not depend on the trivialization on \(M\).

We fix \(p\) and construct first a map \(M\to\P^{N_{\delta_{p}-1}}\) by considering the line \(\ell_{\aa^{p}}(x)\) generated by the vector \((\alpha_{J}(\aa^{p},x))_{J}\). Let 
\[
  \mathcal{Q}^{p}
  \bydef
  \kk^{N_{\delta_{p}}}\diagup\O(-1)
\]
be the universal quotient bundle on \(\P^{N_{\delta_{p}-1}}\).
We then construct a map 
\[
  \Flag_{s}(TM)_{(x,F)}
  \to
  \Flag_{s}(\mathcal{Q}^{p})_{\ell_{\aa^{p}}(x)}
\] 
by considering the flag made of the subspaces \(\Span\left(\Set*{(\theta_{J}(\aa^{p},x,v))_{J}\suchthat v\in F_{i}}\right)\diagup\ell_{\aa^{p}}(x)\), for \(i=1,\dotsc,t\).

In order to proceed rigorously, which is done in Sect.~\ref{se:psi}, one needs to establish that:
\begin{itemize}
  \item
    the evaluation map \((\alpha_{J}(\aa^{p},\cdot))_{J}\) does not vanish on \(M\);
  \item
    \(\forall x\in M\), the map
    \((\theta_{J}(\aa^{p},x,\cdot))_{J}\colon T_{x}M \to\kk^{N_{\delta_{p}}}\diagup\ell_{\aa^{p}}(x)\) is injective.
\end{itemize}
If \(c\) is large enough, this can be achieved simultaneously for all \(p\) for adequate choices of the parameters \(\mathbi{\varepsilon}\) and \(\mathbi{\delta}\), up to shrinking \(S\) into an open dense subset \(S^{\circ}\).

Moreover, we show that the construction actually does not depend on the choice of a trivialization and globalizes.
We hence get a well-defined global morphism
\[
  \Theta
  \colon
  S^{\circ}\times\Flag_{s}(TM)
  \to
  \Flag_{s}(\mathcal{Q}^{1})
  \times\dotsb\times
  \Flag_{s}(\mathcal{Q}^{c}).
\]
This in turn allows one to define a morphism
\[
  \varPsi
  \colon
  \left\vert
  \begin{array}{ccc}
    S^{\circ}
    \times
    \Flag_{s}(TM)
    & \longrightarrow 
    &
    \Flag_{s}(\mathcal{Q}^{1})
    \times\dotsb\times
    \Flag_{s}(\mathcal{Q}^{c})
    \times 
    \P^{N}
    \\
    \left(
      \aa^{\bigcdot}
      ,
      (x,F)
    \right)
    &
    \longmapsto
    &
    \left(\Theta(\aa^{\bigcdot},(x,F)), [\mathbi{\xi}(x)^{r}]\right)
    \\
  \end{array}
\right..
\]
By construction, the restriction of this morphism \(\varPsi\) to \(T_{\X/S^{\circ}}\subset S^{\circ}\times\Flag_{s}(TM)\) factors through the universal family in 
\(\Flag_{s}(\mathcal{Q}^{1}) \times\dotsb\times \Flag_{s}(\mathcal{Q}^{c}) \times \P^{N}\),
defined pointwise as
\[
  \Set*{
    \left((\ell^{1}, F^{1}),\dotsc,(\ell^{c}, F^{c});Z\right)
    \suchthat
    \substack{
      \nu_{\delta_{i}}(Z)\perp \ell^{i}
      \\
      \nu_{\delta_{i}}(Z)\perp F^{i}
    }
  },
\]
where \(\nu_{\delta}\) denote the Veronese embedding of \(\P^{N}\) in \(\P^{N_{\delta}-1}\).

Recall the classical isomorphism of flag bundles 
\[
  \Flag_{s}(\mathcal{Q}^{p})
  \sim
  \Flag_{(s_1+1,\dotsc,s_{t}+1,1)}(\kk^{N_{\delta_{p}}}),
\]
given pointwise plainly by
\[
  (V_{1}\supsetneq\dotsb\supsetneq V_{t};\ell)
  \mapsto
  (V_{1}\oplus\ell\supsetneq\dotsb\supsetneq V_{t}\oplus\ell\supsetneq \ell).
\]
This identification allows one to define forgetful maps 
\[
  \Flag_{s}(\mathcal{Q}^{p}) \to \Flag_{(s_1+1,\dotsc,s_{t}+1)}(\kk^{N_{\delta_{p}}}),
\]
and for technical reasons, we will rather compose \(\Theta\) by these maps and consider the universal family
\[
  \mathcal{Y}
  \subset
  \Flag_{(s_1+1,\dotsc,s_{t}+1)}(\kk^{N_{\delta_{1}}})
  \times\dotsb\times
  \Flag_{(s_1+1,\dotsc,s_{t}+1)}(\kk^{N_{\delta_{c}}})
  \times
  \P^{N},
\]
defined pointwise as
\[
  \Set*{
    \left(F^{1},\dotsc,F^{c};Z\right)
    \suchthat
    \nu_{\delta_{i}}(Z)\perp F^{i}
  }.
\]
The family \(\mathcal{Y}\) is defined by \((1+k)\times c\) equations.
Note that if \((1+k)c\geq N\), like in the Main Theorem, then its generic non-empty fiber in \(\P^{N}\) is finite.

\subsection{The model situation}
\label{sse:model situation}
To sum up, we have obtained the following commutative diagram
\begin{align}
  \label{eqref: diagram outline}
  \xymatrix{
    \Flag_{s}(T_{\X/{S^{\circ}}})
    \ar[rr]^{\varPsi}
    \ar[rrd]^{\Theta}
    \ar[d]
&
&
\mathcal{Y}
\ar[r]
\ar[d]^{\rho}
&
\P^{N}
\\
S^{\circ}
&
&
\mathbi{F}
}
\end{align}
where
\(
\mathbi{F}
\bydef
\Flag_{(s_1+1,\dotsc,s_{t}+1)}(\kk^{N_{\delta_{1}}})
\times\dotsb\times
\Flag_{(s_1+1,\dotsc,s_{t}+1)}(\kk^{N_{\delta_{c}}})
\),
and where \(\rho\) is the first projection.

For each partition \(\mu\) with jumps \((s_1+1,\dotsc,s_{t}+1)\), there is an ample line bundle
\[
  L_{\mu}
  \bydef
  \L_{\mu}(\kk^{N_{\delta_{1}}})
  \boxtimes\dotsb\boxtimes
  \L_{\mu}(\kk^{N_{\delta_{c}}})
\]
on \( \FF \).
We have already seen that the morphism \(\rho\) is generically finite onto its image for codimensions \(c\geq N/(1+k)\).
The pullback \(\rho^{*}L_{\mu}\) of the ample line bundle \(L_{\mu}\) is therefore big and nef.

Moreover, one can obtain a geometric control of the augmented base locus of \(L_{\mu}\) via Nakamaye's theorem:
\begin{equation}
  \label{eq:effectivity}
  \Bs((L_{m\mu}\boxtimes \O_{\P^{N}}(-1))_{| \mathcal{Y}})
  \subset
  \Exc(\rho)
  \qquad
  m\gg1,
\end{equation}
where the exceptional locus \(\Exc(\rho)\) is the reunion of all positive dimensional fibers of the morphism \(\rho\colon\mathcal{Y}\to\mathbi{F}\).
In order to get an effective bound for \(m\), one can use Deng's \og effective Nakamaye's theorem\fg (see \cite{deng2017diverio} and the end of Sect.~\ref{se:psi}).

\subsection{Pulling back the positivity from \texorpdfstring{\(\mathcal{Y}\)}{Y} to \texorpdfstring{\(\mathbi{\Flag_{s}}(\mathbf{T}\mathbi{\X}/_{S^{\circ}})\)}{Flag(TX/S)}}
\label{sse:pulling back positivity}
Let now consider the partition \(\mu\bydef(\lambda_{1},\lambda_{1},\dotsc,\lambda_{k})\).
By construction one has
\[
  \Theta^{*}L_{m\mu}
  =
  \L_{cm\lambda}(T_{\X/S})\otimes\O_{\X}\big(cm(\lambda_{1}+\abs{\lambda})(\abs{\mathbi{\delta}}+\abs{\mathbi{\varepsilon}})\big),
\]
where
\[
  \O_{\X}(1)
  \bydef 
  (\O_{S}\boxtimes \O_{M}(1))\vert_{\X}.
\]
One has then
\begin{equation}
  \label{eq:pullback}
  \varPsi^{*}
  (L_{m\mu}\boxtimes \O_{\P^{N}}(-1))_{| \mathcal{Y}}
  =
  \L_{cm\lambda}(T_{\X/S})\otimes\O_{\X}\big(cm(\lambda_{1}+\abs{\lambda})(\abs{\mathbi{\delta}}+\abs{\mathbi{\varepsilon}})-r\big).
\end{equation}
In other words, the sections of \(L_{m\mu}\) yield sections of Schur powers of the relative cotangent bundle to the family \(\X\). 
Moreover, provided \(r\) is large enough, which can always be done, these sections vanish on an ample divisor.

The next and final step is to investigate the base locus of the sections obtained in this way. This is the object of Sect.~\ref{se:avoiding the exceptional locus}. It is established that for \(\mathbi{\delta}\) large enough, up to schrinking \(S^{\circ}\), any curve \(\mathcal{C}\) in a fiber of \(\Flag_{s}T_{\X/S}\) satisfies \(\varPsi(\mathcal{C}) \not\subset \Exc(\rho)\).
Considering the inclusion \eqref{eq:effectivity}, one infers
\begin{align}
  \label{eq:6}
  \varPsi(\mathcal{C})
  \not\subset
  \Bs((L_{m\mu}\boxtimes \O_{\P^{N}}(-1))_{| \mathcal{Y}}).
\end{align}
In view of \eqref{eq:pullback}, this shows that a negative twist of the (relatively ample) line bundle \(\L_{cm\lambda}(T_{\X/S})\) is nef over its generic fibers, and this leads to the sought ampleness.

This finishes the quick proof of the Main Theorem.
We will now give all the details.

\section{Proof of the main theorem}
\label{se:proof}

\subsection{Reduction to the model situation}
\label{se:psi}
We work with the setting and notation of Sect.~\ref{sse:setting}, on an arbitrary projective variety \((M,\O_{M}(1))\). 

Let us first define rigorously the morphism 
\[
  \Theta
  \colon
  S
  \times
  \Flag_{s}(TM)
  \to
  \Flag_{(s_{1}+1, \dotsc, s_{t}+1)}(\kk^{N_{\delta_{1}}})
  \times\dotsb\times
  \Flag_{(s_{1}+1,\dotsc, s_{t}+1)}(\kk^{N_{\delta_{c}}})
\]
already introduced in Sect~\ref{se:overview}.

\begin{lemma}
  \label{lemma: line}
  For a generic choice of parameter \(\aa^{p}\),
  \[
    \ell_{\aa^{p}}
    \bydef
    \big(
      \alpha_{J}(\aa^{p},\cdot)
    \big)_{J}
    \in 
    \kk^{N_{\delta_{p}}}
  \]
  does not vanish on \(M\) as soon as \(\delta_{p} \geq 1\).
\end{lemma}
\begin{proof}
  This statement follows easily from the very definition of \(\ell_{\aa^{p}} \). Its proof relies on the use of a suitable stratification of the space \(M\), along the lines of the proof of \cite[Proposition 2.6]{BD2018}.
\end{proof}

\begin{lemma}
  \label{lemma: global}
  Let \(\aa^{p} \) be a parameter such that 
  \(\ell_{\aa^{p}}\)
  does not vanish on \(M\). The map \( TM \to \mathcal{Q}^{p} \) locally defined by
  \[
    (x,v)
    \mapsto
    \Big(\ell_{\aa^{p}}(x); \big(\theta_{J}(\aa^{p},x,v)\big)_{J}\Big)
  \]
  extends to a global morphism 
  \( \Theta_{\aa^{p}}\colon
  TM\otimes \O_{M}(\varepsilon_{p}+\delta_{p}) 
  \to
  \mathcal{Q}^p.
  \)

\end{lemma}

\begin{proof}
  Consider two open sets \(V\) and \(V'\) of \(M\) that trivializes the line bundle \(\O_{M}(1)\). Denote by \(g\) the corresponding transition map for \(\O_{M}(1)\). By construction, on the chart \( V\), for any \(J\),
  \[
    \diff\big((\alpha_{J})_{V}(\aa^{p},\cdot)  (\mathbi{\xi})_{V}^{rJ}\big)
    =
    (\theta_{J})_{V}(\aa^{p},\cdot)  (\mathbi{\xi})_{V}^{rJ}.
  \]
  Accordingly, the right-hand side behaves like the left-hand side under a change of chart. Namely on \( V \cap V'\) one has
  \[
    (\theta_{J})_{V}(\aa^{p},\cdot)
    =
    g^{\varepsilon_{p}+\delta_{p}} 
    (\theta_{J})_{V'}(\aa^{p},\cdot)
    +
    \left( (\varepsilon_{p}+(r+1)\delta_{p}) \frac{\diff g}{g}\right) (\alpha_{J})_{V}(\aa^{p},\cdot).
  \]
  One infers that for any \((x,v) \in TM\)
  \[
    \Big(\big(\theta_{J})_{V}(\aa^{p},(x,v)\big)\Big)_{J}
    =
    \Big(\big(\theta_{J})_{V'}(\aa^{p},(x,g^{\varepsilon_{p}+\delta_{p}}(x)v)\big)\Big)_{J} \mod \ell_{\aa^{p}}(x),
  \]
  which yields the result.
\end{proof}

\begin{proposition}
  \label{prop: global}
  Suppose that \(\delta_{p} \geq 2 \). For a generic choice of parameter \(\aa^{p}\), the linear map \(\Theta_{\aa^{p}}\) is injective.
\end{proposition}
\begin{proof}
  The proof of this statement relies on the very ampleness of \(\O_{M}(1)\) and a dimension count (\cite[Lemma 2.5]{BD2018}). It goes again along the same lines as the proof of \cite[Proposition 2.6]{BD2018}.
\end{proof}

By Proposition~\ref{prop: global}, there exists  a dense open subset \(S^{\circ}\) in \(S\) such that for any \((\aa^{1}, \dotsc, \aa^{c})\) in \(S^{\circ}\), all morphisms \( \Theta_{\aa^{p}} \) are injective. We fix such a dense open set \(S^{\circ}\).
We are now in position to define the map \(\Theta\): it is the natural map induced at the level of flag bundles by \(\Theta_{\aa^{1}}, \dotsc, \Theta_{\aa^{p}} \). For technical reasons, we will rather compose these morphisms \(\Theta_{\aa^{p}}\) by the forgetful maps
\[
  \Flag_{s}(\mathcal{Q}^{p})
  \to
  \Flag_{(s_{1}+1, \dotsc, s_{t}+1)} \kk^{N_{\delta_{p}}}.
\]
To sum up, the previous considerations result in the following definition.
\begin{definition}
  Define
  \[
    \Theta\colon
    \left|
    \begin{array}{ccc}
      S^{\circ} \times \Flag_{s}TM  & \longrightarrow & \Flag_{(s_{1}+1, \dotsc, s_{t}+1)}(\kk^{N_{\delta_{1}}}) \times \dotsb \times \Flag_{(s_{1}+1, \dotsc, s_{t}+1)}(\kk^{N_{\delta_{c}}})
      \\
      \big(\aa^{\bigcdot}; (x,F)\big)
                                    &
                                    \longmapsto 
                                    &
                                    \Theta_{\aa^{1}} (x,F)
                                    \times 
                                    \dotsb
                                    \times
                                    \Theta_{\aa^{c}} (x,F)
    \end{array}
  \right. .
\]
\end{definition}
By pull-back, the morphism \(\Theta_{\aa^{\bigcdot}}\bydef\Theta(\aa^{\bigcdot}, \cdot)\) produces positively twisted sections of the Schur power \(S^{\lambda}\Omega_{M}\):
\begin{lemma}
  \label{lemma: pull-back}
  Recall that  \(\mu=(\lambda_{1}, \lambda_{1}, \dotsc, \lambda_{k})\). For \(\aa^{\bigcdot} \in S^{\circ}\),
  the following equality of line bundles on \(\Flag_{s}(TM)\) is satisfied:
  \[
    \Theta_{\aa^{\bigcdot}}^{*}
    L_{\mu} 
    =
    \L_{c\lambda}(TM)
    \otimes
    \pi_{M}^{*}
    \O_{M}\left(
      (\abs{\lambda}+\lambda_{1})(\abs{\mathbi{\varepsilon}} + \abs{\mathbi{\delta}})
    \right),
  \]
  where \( \pi_{M}\colon \Flag_{s}(TM) \to M \) is the canonical projection onto the space \(M\).
\end{lemma}
\begin{proof}
  Note that it is enough to show the following equality for any \(1 \leq p \leq c \)
  \[
    \Theta_{\aa^{p}}^{*}
    \L_{\mu} (\kk^{N_{\delta_{p}}})
    =
    \L_{\lambda}(TM)
    \otimes
    \pi_{M}^{*}
    \O_{M}\left(
      (\abs{\lambda}+\lambda_{1})(\varepsilon_{p} + \delta_{p})
    \right).
  \]
  (Recall that    \( L_{\mu}=\L_{\mu}(\kk^{N_{\delta_{1}}}) \boxtimes \dotsb \boxtimes \L_{\mu}(\kk^{N_{\delta_{c}}})\) ).

  We denote the conjugate partition
  $
  \lambda^{*}=(s_{1}^{b_{1}}, \dotsc, s_{t}^{b_{t}}) ,
  $
  so that
  \(
  \mu^{*}
  =
  \big((s_{1}+1)^{b_{1}}, \dotsc, (s_{t}+1)^{b_{t}}\big)\).
  We fix a basis on \(\kk^{N_{\delta_{p}}}\). The flag variety \(\Flag_{(s_1+1,\dotsc,s_{t}+1)}(\kk^{N_{\delta_{p}}})\) is a closed subvariety of
  \[
    \Grass(s_{1}+1, \kk^{N_{\delta_{p}}})
    \times
    \dotsb
    \times
    \Grass(s_{t}+1, \kk^{N_{\delta_{p}}}),
  \]
  whose coordinates are given by the Plücker coordinates on each factor (defined via the fixed basis).
  The sections of the line bundle \(\L_{\mu} (\kk^{N_{\delta_{p}}})\) are polynomials of multi-degree \((b_{1}, \dotsc, b_{t})\) in the Plücker coordinates (restricted to \(\Flag_{(s_1+1,\dotsc,s_{t}+1)}(\kk^{N_{\delta_{p}}})\)).
  Let
  \[
    s
    =
    (p_{1,1}\dotsm p_{1,b_{1}})
    \times
    \dotsb
    \times
    (p_{t,1}\dotsm p_{t,b_{t}})
  \]
  be a section of \(\L_{\mu}(\kk^{N_{\delta_{p}}})\) associated to a monomial of multi-degree \((b_{1}, \dotsc, b_{t})\), where \(p_{i,j}\) is a Plücker coordinate on \(\Grass(s_{i}+1,\kk^{N_{\delta_{p}}})\). Using e.g. (the proof of) Lemma \ref{lemma: global}, one sees that \(\Theta_{\aa^{p}}^{*}(s)\) is a polynomial of multi-degree \((b_{1}, \dotsc, b_{t})\) in the Plücker coordinates on \(\Grass(s_{1}, TM) \times\dotsb\times \Grass(s_{t}, TM)\) (restricted to \(\Flag_{s}(TM)\)), and it takes values in the line bundle
  \[
    \pi_{M}^{*}
    \O_{M}
    \left(\big(b_{1}(s_{1}+1)+\dotsb+b_{t}(s_{t}+1)\big)(\varepsilon_{p}+\delta_{p})\right)
    =
    \pi_{M}^{*}
    \O_{M}
    \left((\abs{\lambda}+\lambda_{1})(\varepsilon_{p}+\delta_{p})\right).
  \]
  It is therefore a section of \(\L_{\lambda}(TM) \otimes \pi_{M}^{*}\O_{M}\left((\abs{\lambda}+\lambda_{1})(\varepsilon_{p}+\delta_{p})\right)\).
  To deduce the equality between the line bundles, it is enough to observe that \(\L_{\mu} (\kk^{N_{\delta_{p}}})\) is globally generated (it is in fact very ample).
\end{proof}

Recall that we introduced in Sect.~\ref{se:overview} the morphism
\[
  \varPsi
  \colon
  \left\vert
  \begin{array}{ccc}
    S^{\circ}
    \times
    \Flag_{s}(TM)
    & \longrightarrow 
    &
    \FF
    \times 
    \P^{N}
    \\
    \left(
      \aa^{\bigcdot}
      ,
      (x,F)
    \right)
    &
    \longmapsto
    &
    \left(\Theta(\aa^{\bigcdot},(x,F)), [\mathbi{\xi}(x)^{r}]\right)
    \\
  \end{array}
\right.,
\]
where we have denoted 
\[
  \FF
  =
  \Flag_{(s_{1}+1, \dotsc, s_{t}+1)}(\kk^{N_{\delta_{1}}}) \times \dotsb \times \Flag_{(s_{1}+1, \dotsc, s_{t}+1)}(\kk^{N_{\delta_{c}}}).
\]
The restriction of the morphism \(\varPsi\) to the flag bundle \( \Flag_{s}(T_{\X/S^{\circ}}) \) factors through the universal family
\[
  \mathcal{Y}
  =
  \Set*{
    \left(F^{1},\dotsc,F^{c};Z\right) \in \FF \times \P^{N}
    \suchthat
    \nu_{\delta_{i}}(Z)\perp F^{i}
  }.
\]
The following commutative diagram sums up the situation.
\[
  \xymatrix{
    \Flag_{s}(T_{\X/{S^{\circ}}})
    \ar[rr]^{\varPsi}
    \ar[rrd]^{\Theta}
    \ar[d]
&
&
\mathcal{Y}
\ar[r]
\ar[d]^{\rho}
&
\P^{N}
\\
S^{\circ}
&
&
\mathbi{F}
}
\]
As soon as \((1+k)c \geq N \), the first projection \(\rho\) is generically finite onto its image, so that the pull-back by  \(\rho\) of the very ample line bundle \(L_{\mu}\) on \(\FF\) is big and nef. Using Lemma \ref{lemma: pull-back}, one has for any \(m \geq 1\)
\begin{equation}
  \label{eq:pullback1}
  \varPsi^{*}
  (L_{m\mu}\boxtimes \O_{\P^{N}}(-1))_{| \mathcal{Y}}
  =
  \L_{cm\lambda}(T_{\X/S^{\circ}})\otimes\O_{\X}\big(cm(\lambda_{1}+\abs{\lambda})(\abs{\mathbi{\delta}}+\abs{\mathbi{\varepsilon}})-r\big).
\end{equation}
For \(m \gg 1\), the base locus of the line bundle bundle 
\(
(L_{m\mu}\boxtimes \O_{\P^{N}}(-1))_{| \mathcal{Y}}
\)
is nothing but the augmented base locus of \(\rho^{*}L_{\mu}\). By Nakamaye's theorem, this is the exceptional locus  \( \FF_{\infty} \bydef \Exc(\rho)\) of the morphism \( \rho \). Regrettably, this theorem does not provide an explicit bound for the number \(m\). For our purpose, it is enough to obtain a value of \(m\) so that the base locus \( \Bs\left((L_{m\mu}\boxtimes \O_{\P^{N}}(-1))_{| \mathcal{Y}}\right) \) is included in the exceptional locus \( \FF_{\infty} \).

This technical work on the effectivity of the value \(m\) was performed by Deng in~\cite{deng2017diverio}, by treating each natural stratum of \(\P^{N}\) at a time. More precisely,
denote for \(I \subsetneq \{0,\dotsc, N\}\)
\[
  \P(\kk_{I})
  \bydef
  \{
    (z_{0},\dotsc, z_{N})
    \in
    \P^{N}
    \
    |
    \
    \forall i \in I,
    \ z_{i}=0
  \}.
\]
Denote 
\(
\mathcal{Y}(I)
\bydef 
\mathcal{Y} \cap \big(\FF \times \P(\kk_{I})\big)
\),
as well as  \(\FF_{\infty}(I)\) the exceptional locus of the first projection \(\rho\) restricted to \(\FF \times \P(\kk_{I})\).
An easy corollary of a result in~\cite[Theorem 2.10]{deng2017diverio} yields the following
\begin{proposition}
  \label{prop: naka}
  For any \( I \subsetneq \{0, \dotsc, N\} \) and for any integer \(m  \geq \frac{\prod_{j=1}^{c} \delta_{j}^{k+1}}{\underset{1 \leq i \leq c}{\min}\delta_{i}} \), we have the inclusion
  \begin{align*}
    \rho\left(
      \Bs(L_{m\mu}
      \boxtimes
      \O_{\P^{N}}(-1)_{| \mathcal{Y}(I)})
    \right)
    \subset
    \FF_{\infty}(I),
  \end{align*}
  where \(\mu\) is any partition with jump sequence \((s_{1}+1, \dotsc, s_{t}+1)\).
\end{proposition}
\begin{proof}
  Denote \(\GG\bydef\Grass(k+1, \kk^{N_{\delta_{1}}})\times \dotsb \times \Grass(k+1, \kk^{N_{\delta_{c}}})\). One has the commutative diagram for any \(I\)
  \begin{align*}
    \xymatrix{
      \FF
      \ar[d]^{q}
   &
   &
   \FF\times \P(\kk_{I})
   \ar[ll]^{\rho}
   \ar[d]^{q \times \Id}
   \\
   \GG
   &
   &
   \GG\times \P(\kk_{I})
   \ar[ll]^{\rho^{*}}
 }
\end{align*}
where \(q\) is the natural projection from \(\FF\) to \(\GG\) obtained by sending on each factor a flag \((F_{1} \supsetneq \dotsb \supsetneq F_{t})\) to its first vector space \(F_{1}\).
Recall that we defined, for \( m \in \N_{\geq 1} \), the very ample line bundle on \(\GG\)
\[
  L_{m^{k+1}}
  \bydef
  \O_{\Grass(k+1, \kk^{N_{\delta_{1}}})}(m)
  \boxtimes
  \dotsb
  \boxtimes
  \O_{\Grass(k+1, \kk^{N_{\delta_{c}}})}(m).
\]
Denote \(\mathcal{Y}^{*}(I)\bydef(q \times \Id)(\mathcal{Y}(I))\).
A theorem in~\cite[Theorem 2.10]{deng2017diverio} asserts that, if
\(m \geq \frac{\prod_{j=1}^{c} \delta_{j}^{k+1}}{\underset{1 \leq i \leq c}{\min}\delta_{i}}\),
one has the following inclusion
\[
  \Bs(
  L_{m^{k+1}}
  \boxtimes
  \O_{\P^{N}}(-1)_{| \mathcal{Y^{*}}(I)})
  \subset
  (\rho^{*})^{-1}
  \big(
    \GG_{\infty}(I)
  \big),
\]
where
\(
\GG_{\infty}(I)
\)
is the exceptional locus of  \(\rho^{*}\). By pull-back, and by the very definitions of the line bundles \(L_{m^{k+1}}\) and \(L_{m\mu}\), one infers that
\[
  \Bs(L_{m\mu}
  \boxtimes
  \O_{\P^{N}}(-1)_{| \mathcal{Y}(I)})
  \subset
  \rho^{-1}(\FF_{\infty}(I)),
\]
which concludes the proof.
\end{proof}

\subsection{Avoiding the exceptional locus}
\label{se:avoiding the exceptional locus}
For technical reasons, we now introduce a stratification of the space \( M \) according to coordinate hyperplanes:
\begin{definition}
  For \(I \subsetneq \Set{0,\dotsc,N}\),
  denote
  \(D_{I}\bydef\bigcap_{i \in I} D_{i}\),
  where
  \(D_{i}\bydef\{\xi_{i} = 0 \}\),
  as well as
  \[
    M_{I}
    \bydef
    D_{I} \setminus \bigcap_{i \notin I} D_{i}.
  \]
  If \(\abs{I}=N-k_{0}\), \(\dim  M_{I}=k_{0}\),  and the subspaces \(M_{I}\) stratify \(M\).
\end{definition}

For a generic choice of parameter \(\aa^{\bigcdot}\), the restriction of the morphism \(\Theta(\aa^{\bigcdot},\cdot)\) to a positive dimensional stratum avoids the exceptional locus \(\FF_{\infty}(I)\):
\begin{proposition}
  \label{prop: exc}
  If \(\delta_{1},\dotsc, \delta_{c}\) are \(\geq N+k(N-k)\), up to shrinking \(S^{\circ}\), one can assume that for each \(I \subsetneq \{0,\dotsc, N\}\) with \(\abs{I} < N\),
  \[
    \Theta\left(
      S^{\circ}
      \times
      \Flag_{s}(TM_{| M_{I}})
    \right)
    \cap
    \FF_{\infty}(I)
    =
    \emptyset.
  \]

\end{proposition}
\begin{proof}
  Recall that \(\GG=\Grass(k+1, \kk^{N_{\delta_{1}}})\times \dotsb \times \Grass(k+1, \kk^{N_{\delta_{c}}})\). One has the commutative diagram
  \begin{align*}
    \xymatrix{
      S^{\circ} \times \Flag_{s}(TM_{| M_{I}})
      \ar[rrr]^{\Theta}
      \ar[d]^{\Id \times p}
    &
    &
    &
    \FF
    \ar[d]^{q}
    &
    &
    \FF\times \P(\kk_{I})
    \ar[ll]^{\rho}
    \ar[d]^{q \times \Id}
    \\
    S^{\circ} \times \Grass(k, TM_{| M_{I}})
    \ar[rrr]^{\Theta}
    &
    &
    &
    \GG
    &
    &
    \GG\times \P(\kk_{I})
    \ar[ll]^{\rho^{*}}
  }
\end{align*}
where the map \(p\) is the projection given pointwise by \( \big(x, F=(F_{1}\supsetneq \dotsb \supsetneq F_{t})\big) \mapsto (x,F_{1}) \), and the map \(q\) is the projection constructed in a similar fashion. Accordingly, one sees that it is enough to prove the following: up to shrinking \(S^{\circ}\),  for each \(I \subsetneq \{0,\dotsc, N\}\) with \(\abs{I} < N\),
\[
  \Theta\left(S^{\circ}
    \times
    \Grass(k+1,TM_{| M_{I}})
  \right)
  \cap
  \GG_{\infty}(I)
  =
  \emptyset,
\]
where \(\GG_{\infty}(I)\) is the exceptional locus of the first projection \(\rho^{*}\) (restricted to \(\GG \times \P(\kk_{I})\)).
To this end, it is enough to prove the following inequality
\[
  \dim
  \big(
    (S^{\circ} \times \Grass(k, TM_{| M_{I}}))
    \cap
    \Theta^{-1}(\GG_{\infty}(I))
  \big)
  <
  \dim S^{\circ}.
\]
(Note that, here,  \(\Theta\) is implicitely the restriction of \(\Theta\) to the stratum \(\Grass(k, TM_{|M_{I}})\)), as indicated by the commutative diagram above).

Fix \(\eta=(x, [v_{1}\wedge \dotsb \wedge v_{k}]) \in \Grass(k, TM_{| M_{I}})\), and denote \(\Theta_{\eta}\bydef \Theta(\cdot, \eta)\).
Note that if one proves the following inequality

\[
  \dim
  \big(
    \Theta_{\eta}^{-1}(
    \GG_{\infty}(I)
    )
  \big)
  <
  \dim S^{\circ}
  -
  \dim \left(\Grass(k, TM_{| M_{I}})\right),
\]
the lemma will follow, since then
\begin{align*}
&
\dim
\left(
  (S^{\circ}
  \times
  \Grass(k, TM_{| M_{I}}))
  \cap
  \Theta^{-1}(\GG_{\infty}(I))
\right)
\\
&
\leq
\dim \left(\Grass(k, TM_{| M_{I}})\right)
+
\underset{\eta }{\max}
\dim
\left(
  \Theta_{\eta}^{-1}(
  \GG_{\infty}(I)
  )
\right)
\\
&
<
\dim S^{\circ}.
\end{align*}

In order to study
\(\dim \big( \Theta_{\eta}^{-1}( \GG_{\infty}(I)) \big)\),
recall that  
\[
  \Theta_{\eta}(\aa^{\bigcdot})=\Theta_{\eta}(\aa^{1}) \times \dotsb \times \Theta_{\eta}(\aa^{c}) ,
\]
where \(\Theta_{\eta}(\aa^{p})\) is the \((k+1)\)-plane in \(\kk^{N_{\delta_{p}}}\) spanned by the components of the vector
\[
  \varphi_{\eta}(\aa^{p})
  \bydef
  \left(\big(\alpha_{J}(\aa^{p}, x)\big)_{J}, \big(\theta_{J}(\aa^{p},v_{1})\big)_{J}, \dotsc, \big(\theta_{J}(\aa^{p},v_{k})\big)_{J}\right).
\] 
One can see this vector as belonging to \(H^{0}(\P^{N}, \O_{\P^{N}}(\delta_{p}))^{\times (k+1)}\), and since \(x \in M_{I}\), one rather composes the linear map \(\varphi_{\eta}\) by the product of the natural projection
\[
  H^{0}(\P^{N}, \O_{\P^{N}}(\delta_{p}))
  \to
  H^{0}(\P(\kk_{I}), \O_{\P(\kk_{I})}(\delta_{p})).
\]
Accordingly, the vector \(\varphi_{\eta}(\aa^{p})\) belongs to \( H^{0}(\P(\kk_{I}), \O_{\P(\kk_{I})}(\delta_{p}))^{\times k+1} \).
Denote
\[
  \overset{\sim}{\GG}_{\infty}(I)
  \bydef
  H^{0}(\P(\kk_{I}), \O_{\P(\kk_{I})}(\delta_{1}))^{\times (k+1)}
  \times
  \dotsb
  \times
  H^{0}(\P(\kk_{I}), \O_{\P(\kk_{I})}(\delta_{c}))^{\times (k+1)}
\]   
and form 
\[
  \varPhi_{\eta}(\aa^{\bigcdot})
  \bydef
  \varphi_{\eta}(\aa^{1})
  \times
  \dotsb 
  \times 
  \varphi_{\eta}(\aa^{p})
  \colon S \to \overset{\sim}{\GG}_{\infty}(I).
\]
With this point of view, one sees that for \(Z \in \P({\kk_{I}})\) and \(\aa^{\bigcdot} \in S^{\circ}\), we have an equivalence between the vanishing of  \(\varPhi_{\eta}(\aa^{\bigcdot})(Z)\) and the conditions \(\Big(\nu_{\delta_{p}}(Z)\perp \Theta_{\eta}(\aa^{p}) \ \forall 1 \leq p \leq c \Big)\). (Recall that  \(\nu_{\delta_{p}}\) denotes the Veronese embedding of \(\P^{N}\) in \(\P^{N_{\delta_{p}}-1}\)).
This leads to define the analogue universal family
\[
  \overset{\sim}{\mathcal{Y}}(I)
  \bydef
  \Set*{\left((P_{i,j})_{i,j}, Z\right) \in \overset{\sim}{\GG}(I)\times \P(\kk_{I})
    \suchthat 
    P_{i,j}(Z)=0 \ 
    \forall 
    \
    i, j
  } 
\]    
with associated first projection \(\overset{\sim}{\rho}\). If \(\overset{\sim}{\GG}_{\infty}(I)\) is the exceptional locus of \(\overset{\sim}{\rho}\), one deduces from the above that for \( \aa^{\bigcdot} \in S^{\circ} \)
\[
  \Theta_{\eta}(\aa^{\bigcdot}) \in \GG_{\infty}(I) 
  \ 
  \text{if and only if}
  \
  \varPhi_{\eta}(\aa^{\bigcdot}) \in \overset{\sim}{\GG}_{\infty}(I).
\]
This gives in turn that
\begin{align*}
  \dim \big( \Theta_{\eta}^{-1}( \GG_{\infty}(I)) \big)
&
=
\dim \big( \varPhi_{\eta}^{-1}( \overset{\sim}{\GG}_{\infty}(I)) \big)
\\
&
=
\dim(S)-\rank(\varPhi_{\eta})+\dim\big(\overset{\sim}{\GG}_{\infty}(I)\big).
\end{align*}
It remains to evaluate the rank of \(\varPhi_{\eta}\). Suppose for the moment that \(\varPhi_{\eta}\) is surjective. Therefore, one gets the equality
\[
  \dim \big( \Theta_{\eta}^{-1}( \GG_{\infty}(I)) \big)
  =
  \dim(S)-\codim(\overset{\sim}{\GG}_{\infty}(I), \overset{\sim}{\GG}(I)).
\]
The quantity \(\codim(\overset{\sim}{\GG}_{\infty}(I), \overset{\sim}{\GG}(I))\) was studied in~\cite[Cor 3.3]{BD2018} as a corollary of a result in~\cite{Benoit}, and one has that
\[
  \codim(\overset{\sim}{\GG}_{\infty}(I), \overset{\sim}{\GG}(I))
  >
  \underset{1 \leq p \leq c}{\min} (\delta_{p}).
\]
Thus, if the \(\delta_{p}\)'s are chosen greater than \(N+k(N-k)=\dim \left( \Grass(k, TM) \right) \geq \dim \left(\Grass(k, TM_{| M_{I}})\right)\), the lemma follows.

To complete the proof, it remains to show that \(\varPhi_{\eta}\) is surjective. To this end, it is enough to show that 
\[
  \varphi_{\eta}\colon S_{\varepsilon_{p}, \delta_{p}} \to H^{0}(\P(\kk_{I}), \O_{\P(\kk_{I})}(\delta_{p}))^{\times (k+1)} 
\]
is surjective for \(1 \leq p \leq c\).
The matrix associated to the map \(\varphi_{\eta}\) can be represented by blocks
\begin{equation*}
  \renewcommand*{\arraystretch}{2}
  \def\mast{{\renewcommand*{\arraystretch}{.8}\begin{bmatrix}\ast&\ast\\\ast&\ast\end{bmatrix}}}
  \def\zero{{\renewcommand*{\arraystretch}{.8}\begin{bmatrix}0&0\\0&0\end{bmatrix}}}
  \begin{pmatrix}
    \mast&\zero&\dotso&\zero\\
    \zero&\ddots&\ddots&\vdots\\
    \vdots&\ddots&\ddots&\zero\\
    \zero&\dotso&\zero&\mast\\
  \end{pmatrix}
\end{equation*}
where each block is a matrix with \(k+1\) rows and \(\dim(H^{0}(M, \O_{M}(\varepsilon)))\) columns, and the number of such blocks on a column is equal to \(\dim H^{0}(\P(\kk_{I}), \O_{\P(\kk_{I})}(\delta_{p}))\).
To see this, one considers
\(S\)
as
\((H^{0}(M,\O_{M}(\varepsilon)))_{\abs{J}=\delta_{p}}\)
and
\(H^{0}(\P(\kk_{I}), \O_{\P(\kk_{I})}(\delta_{p}))^{\times (k+1)} \)
as
\((\kk^{k+1})_{[J]\cap I=\emptyset, \abs{J}=\delta_{p}}\), where \( [J] \) is the support of the index \(J\).
Each vertical entity in the previous block representation corresponds to the image of \(\{0\}\times \dotsb \times H^{0}(M,\O_{M}(\varepsilon))\times \{0\}\times \dotsb \times \{0\}\), where the non zero-term corresponds to a certain index \(J\).

Take an index \(J\), with \([J] \cap I = \emptyset\). Under the hypothesis \(\varepsilon_{p} \geq 1\), every element of \(\Omega_{x}M\) can be written under the form \(\diff a(x,\cdot)\) where \(a \in H^{0}(M,\O_{M}(\varepsilon))\) satisfies \(a(x)=0\).
Letting \(\aa^{p}\) be such that \(a_{J}(x)=0\), one has for every \(1\leq i \leq k\):
\[
  \theta_{J}(\aa^{p}, x, v_{i})
  =
  \mathbi{\xi}^{J}(x) \diff a_{J}(x, v_{i}).
\]
From the previous observation, for every \(1 \leq i \leq k\), one can find \(\mathbf{a_{i}^{p}}\) such that
\[
  \theta_{J}(\mathbf{a_{i}^{p}}, x, v_{j})=\delta_{i}^{j}
\]
for every \(1 \leq j \leq k\), where \(\delta_{i}^{j}\) is the Kronecker index.
As one can always find
\( 
a \in H^{0}(M,\O_{M}(\varepsilon)) 
\)
such that \(a(x)\neq 0\), one deduces that the rank of the \(J\)th block is maximal, i.e. equal to \(k+1\).
This reasoning works for every index \(J\), with \([J] \cap I = \emptyset\), and shows that the rank of the matrix is at least equal to \((k+1) \abs{\Set*{J \suchthat [J] \cap I = \emptyset}}=\dim H^{0}(\P(\kk_{I}), \O_{\P(\kk_{I})}(\delta_{p}))^{\times (k+1)} \): the surjectivity of the map \(\varPhi_{\eta}\) follows, and the proof is complete.
\end{proof}

\subsection{Pulling-back the positivity}
We are now in position to finish the proof of Main Theorem. We first prove the following more precise version. 
\begin{theorem}
  \label{thm: main theorem}
  Let \(\lambda\) a partition of jump sequence \(s=(k=s_{1}>\dotsb>s_{t})\), and let \(M\) be a variety of dimension \(N\) equipped with a very ample line bundle \(\O_{M}(1)\).

  Let us choose:
  \begin{itemize}
    \item
      \(c\) such that \(c(s_{1}+1) \geq N \Leftrightarrow c(k+1)\geq N\);
    \item
      \(\mathbi{\delta}=(\delta_{1},\dotsc,\delta_{c})\) such that for each \(1 \leq i \leq c\), \(\delta_{i} \geq N+k(N-k)\);
    \item
      \( m \) such that \( m \geq \frac{\prod_{j=1}^{c} \delta_{j}^{k+1}}{\underset{1 \leq i \leq c}{\min}\delta_{i}}\);
    \item
      \(\mathbi{\varepsilon}=(\varepsilon_{1},\dotsc,\varepsilon_{c})\) such that for each \(1 \leq i \leq c\), \(\varepsilon_{i} \geq 1\);
    \item
      \(r\) such that \(r>m(\abs{\lambda}+\lambda_{1})(\abs{\mathbi{\varepsilon}}+\abs{\mathbi{\delta}})\).
  \end{itemize}

  Let us denote for each \(1 \leq i \leq c\), \(d_{i}=\delta_{i}(r+1)+\varepsilon_{i}\).
  Then a complete intersection \(X\bydef H_{1}\cap\dotsb\cap H_{c}\) of \(c\) generic hypersurfaces \(H_{i} \in |H^{0}(M, \O_{M}(d_{i}))|\) has \(\S^{\lambda}\Omega_{X}\) ample.
\end{theorem}

\begin{proof}
  By the very hypothesis, Propositions~\ref{prop: global}, \ref{prop: naka} and \ref{prop: exc} apply.
  Denote \(u=r-m(\abs{\lambda}+\lambda_{1})(\abs{\mathbi{\varepsilon}}+\abs{\mathbi{\delta}}) > 0\).
  One first proves that for  \(\aa^{\bigcdot} \in S^{\circ}\), the line bundle
  \[
    \L_{mc\lambda}(TX_{\aa^{\bigcdot}}) \otimes \pi_{\aa^{\bigcdot}}^{*}\O_{X_{\aa^{\bigcdot}}}(-u)
  \]
  is nef. (Recall that \( \pi_{\aa^{\bigcdot}}: \Flag_{s}(TX_{\aa^{\bigcdot}}) \to X_{\aa^{\bigcdot}} \) is the canonical projection onto  the complete intersection \(X_{\aa^{\bigcdot}}= \pi_{\X}^{-1}(\aa^{\bigcdot})\)).

  Take \(\mathcal{C} \subset \Flag_{s}TX_{\aa^{\bigcdot}} \subset \Flag_{s}TM\) any irreducible curve. As \((\Flag_{s}TM_{| M_{I}})_{I \subsetneq \{0,\dotsc, N\}}\) stratifies \(\Flag_{s}(TM)\), there exists a unique set \(I\) such that \(\Flag_{s}TM_{| M_{I}}\) contains an open dense subset of \(\mathcal{C}\). If \(\dim(M_{I})=0\), then the irreducible curve \(\mathcal{C}\) is contained in a fiber of
  \(\Flag_{s}(TX_{\aa^{\bigcdot}}) \to X_{\aa^{\bigcdot}}\),
  say the fiber over a point \(x \in X_{\aa^{\bigcdot}}\). As \(\L_{\lambda}(T_{x}X_{\aa^{\bigcdot}})\) is ample, one indeed infers the inequality
  \[
    \mathcal{C}
    \cdot
    \L_{mc\lambda}(TX_{\aa^{\bigcdot}})
    \otimes
    \pi_{X_{\aa^{\bigcdot}}}^{*}
    \O_{X_{\aa^{\bigcdot}}}(-u)
    \geq
    0.
  \]
  Suppose now that \(\dim(M_{I})=k \geq 1\), and denote \(\mathcal{C}_{| I}\) the restriction of \(\mathcal{C}\) to \(\Flag_{s}TM_{| M_{I}}\). By Proposition~\ref{prop: exc}, one has
  \[
    \mathcal{C}_{| I}
    \not\subset
    {\varPsi_{I}}^{-1}
    \big(
      \rho_{I}^{-1}(\FF_{\infty}(I))
    \big),
  \]
  where \(\varPsi_{I}\) is the restriction of \(\varPsi\) to the stratum \(\Flag_{s}(TM_{|M_{I}})\), and \(\rho_{I}\) the restriction of \(\rho\) to \(\FF \times \P(\kk_{I})\). (Recall that \( \varPsi\circ \rho=\Theta \)).
  A fortiori, by Proposition~\ref{prop: naka}, one deduces that
  \[
    \mathcal{C}_{| I}
    \not\subset
    {\varPsi_{I}}^{-1}
    \big(
      \Bs(L_{m\mu}
      \boxtimes
      \O_{\P^{N}}(-1)_{| \mathcal{Y}(I)})
    \big).
  \]
  Therefore, there exists a section
  $\sigma \in H^{0}
  \big(
    \mathcal{Y}(I),
    L_{m\mu}
    \boxtimes
    \O_{\P^{N}}(-1)_{| \mathcal{Y}(I)}
  \big)$ such that
  \[
    \mathcal{C}_{| I} \not\subset ({\varPsi_{I}}^{*}\sigma=0).
  \]
  With Lemma~\ref{lemma: pull-back} (more precisely, with equation \eqref{eq:pullback1}), one sees that \(\varPsi_{I}^{*}\sigma\)  is a section of the line bundle (restricted to the stratum \(S^{\circ} \times \Flag_{s}(TM_{|M_{I}})\))
  \[
    \L_{cm\lambda}(T_{\X/S^{\circ}})
    \otimes\O_{\X}(-u),
  \]
  so that
  \(
  \mathcal{C}_{| I}
  \cdot
  \L_{cm\lambda}(T_{\X/S^{\circ}})
  \otimes\O_{\X}(-u)
  \geq
  0.
  \)
  Since \(\mathcal{C} \subset \Flag_{s}TX_{\aa^{\bigcdot}}\), it implies in turn that
  \[
    \mathcal{C}
    \cdot
    \L_{mc\lambda}(TX_{\aa^{\bigcdot}}) 
    \otimes 
    \pi_{\aa^{\bigcdot}}^{*}\O_{X_{\aa^{\bigcdot}}}(-u)
    \geq
    0.
  \]
  The announced nefness is proved.

  To conclude, one sees that the line bundle
  \(
  \left(\L_{\lambda}(TX_{\aa^{\bigcdot}})\right)^{ mc}
  =
  \L_{mc\lambda}(TX_{\aa^{\bigcdot}})
  \)
  is relatively ample with respect to
  \(\pi_{X_{\aa^{\bigcdot}}}\),
  and equal to the sum of a nef divisor and the pull-back by \(\pi_{\aa^{\bigcdot}}\) of an ample divisor.
  Accordingly, it is an ample line bundle, and so is
  \(\L_{\lambda}(TX_{\aa^{\bigcdot}})\).
  By Proposition~\ref{prop:LN}, one infers that the Schur power \(\S^{\lambda}\Omega_{X_{\aa^{\bigcdot}}}\) is ample on \(X_{\aa^{\bigcdot}}\), and the result now follows from the open property in family of ampleness.
\end{proof}

Note the specific constraint on the degrees that prevents from embracing every large enough degrees. We will now see how, following the idea of \og product coup\fg \ of Xie in~\cite{Xie2018}, it is possible to overcome this.
Let us state the key observation in~\cite[Observation 5.6]{Xie2018}
\begin{lemma}
  \label{lemma: uniform bound}
  For all positive integers \(d\geq 1\), every integer \(d_{0}\geq d(d+1)\) is a sum of non-negative multiples of \(d+1\) and \(d+2\).
\end{lemma}

Denote \(\delta=N+k(N-k)\),  fix \(\delta_{p}=\delta\) for every \(1\leq p \leq c\), and fix \(m=\delta^{c(k+1)-1}\). Fix also \(r=\big(2c(\abs{\lambda}+\lambda_{1})\delta^{c(k+1)}-1\big)\), so that
\(r>m(\abs{\lambda}+\lambda_{1})(\abs{\mathbi{\varepsilon}}+\abs{\mathbi{\delta}})\) for every choice of \(\varepsilon_{p} \in \{1,2\}\).
Finally, denote \(d=\delta(r+1)\).
We now prove our Main Theorem, namely that a complete intersection \(X\) of \(c\) generic hypersurfaces
\(H_{i} \in |H^{0}(M, \O_{M}(d_{i}))|\)
with \(d_{i} \geq (d+1)^{2}\) has \(\S^{\lambda}\Omega_{X}\) ample:
\begin{theorem}
  \label{thm: main theorem2}
  Let \(\lambda\) a partition of jump sequence \(s=(s_{1}\bydef k >\dotsb>s_{t})\), and let \(M\) be a variety of dimension \(N\) equipped with a very ample line bundle \(\O_{M}(1)\).
  A complete intersection \(X\bydef H_{1}\cap\dotsb\cap H_{c}\) of \(c\) generic hypersurfaces \(H_{i} \in |H^{0}(M, \O_{M}(d_{i}))|\) with $c \geq \frac{N}{k+1}$ and \(d_{i} \geq \left(1+ 2c(\abs{\lambda}+\lambda_{1})(N+k(N-k))^{c(k+1)+1}\right)^{2}\) has \(\S^{\lambda}\Omega_{X}\) ample.
\end{theorem}
\begin{proof}
  Fix \(d_{1}, \dotsc, d_{c}\) with \(d_{i} \geq (d+1)^{2} \geq d(d+1)\), and, using Lemma~\ref{lemma: uniform bound}, write \(d_{i}=p_{i}(d+1)+q_{i}(d+2)\), with \(p_{i}, q_{i} \in \N\). The idea is to consider, for \(1\leq i \leq c\), hypersurfaces of degree \(d_{i}\) defined by
  \[
    H_{i}
    =
    (h_{i,1}\dotsm h_{i,p_{i}})
    \times
    (h_{i,p_{i}+1}\dotsm h_{i,p_{i}+q_{i}})
  \]
  where \(h_{i,1}, \dotsc,  h_{i,p_{i}} \in H^{0}(M,\O_{M}(d+1))\)
  and
  \(h_{i,p_{i}+1}, \dotsc, h_{i,p_{i}+q_{i}} \in  H^{0}(M,\O_{M}(d+2))\).

  Let \(\big(h_{i,j}\big)_{1\leq i \leq c}^{1\leq j \leq p_{i}+q_{i}}\) be a generic choice of equations, and let 
  \[
    X=(H_{1}=0)\cap \dotsb \cap (H_{c}=0).
  \]
  If we restrict ourselves to regular points \(X_{reg}\) of \(X\), each pair of equations
  \[
    H_{i}(x)=0, \ \diff H_{i}((x,\cdot))=0
  \]
  decomposes as
  \[
    h_{i,k}(x)=0, \ \diff h_{i,k}((x,\cdot))=0
  \]
  for a unique  \(1\leq k \leq p_{i}+q_{i}\). Indeed, the point \(x\) can not belong to \((h_{i,k}=0)\cap(h_{i,k'}=0)\) for \(k\neq k'\) since this would imply that \(x\) is a singular point of \(X\).
  Accordingly, one can decompose \(\Flag_{s}TX_{reg}\) as the disjoint union
  \begin{align}
    \label{eq:decomposition}
    \bigsqcup_{1\leq k_{i} \leq p_{i}+q_{i}} \Flag_{s}(T(X_{k_{1}, \dotsc, k_{c}}\cap X_{reg}))
  \end{align}
  where \(X_{k_{1}, \dotsc, k_{c}}=(h_{1,k_{1}}=0)\cap \dotsb \cap (h_{c, k_{c}}=0)\).
  By Theorem \ref{thm: main theorem}, for a generic choice of equations, the complete intersection \(X_{k_{1}, \dotsc, k_{c}}\) is smooth and the line bundle
  \(\L_{\lambda}(TX_{k_{1}, \dotsc, k_{c}})\)
  is ample for every \(k_{1}, \dotsc, k_{c}\).

  Denote
  \(\X \to T\)
  the universal family of complete intersections of \(c\) hypersurfaces of degrees  \(d_{1}, \dotsc, d_{c}\),
  and denote \(\X^{reg}\)  the open subset whose fiber over \(t \in T\) is the regular points of the complete intersection \(X_{t}\) determined by \(t\).
  Consider the closure of the flag bundle of the relative tangent bundle of \(\X^{reg}/T\)
  and
  denote
  \(\pr_{1}\) (resp. \(\pr_{2}\))
  the restriction to 
  \[
    \overline{\Flag_{s}(T_{\X^{reg}/T})}  \subset T \times \Flag_{s}TM 
  \]
  of the first (resp. the second) projection.
  If the parameter \(t \in T\) defines a smooth complete intersection \(X_{t}\), then
  \[
    \pr_{1}^{-1}(t)
    =
    \Flag_{s}TX_{t},
  \]
  since the fiber \(\Flag_{s} TX_{t}\) is already closed.
  Consider now \(t_{0} \in T\) the parameter such that
  \[
    X_{t_{0}}=X=(H_{1}=0)\cap \dotsb \cap (H_{c}=0).
  \]
  In this situation, it follows from \eqref{eq:decomposition} that
  \[
    \pr_{1}^{-1}(t_{0})
    =
    \bigcup_{1\leq k_{i} \leq p_{i}+q_{i}} \Flag_{s}(TX_{k_{1}, \dotsc, k_{c}}).
  \]

  We can now conclude. Let \(L=\pr_{2}^{*}\L_{\lambda}(TM)\), and denote \(L_{t}\bydef L_{| \pr_{2}^{-1}(t)}\).
  For a parameter \(t \in T\) such that \(X_{t}\) is smooth,
  \[
    L_{t}=\L_{\lambda}(TX_{t}),
  \]
  whereas for the parameter \(t_{0}\) fixed above
  \[
    (L_{t_{0}})_{| \Flag_{s}TX_{k_{1}, \dotsc, k_{c}}}=\L_{\lambda}(TX_{k_{1}, \dotsc, k_{c}}).
  \]
  The very choice of the parameter \(t_{0}\) makes \(\L_{\lambda}(TX_{k_{1}, \dotsc, k_{c}})\) ample for every \(k_{1}, \dotsc, k_{c}\). Accordingly, the line bundle \(L_{t_{0}}\) is ample on every irreducible component of \(\pr_{1}^{-1}(t_{0})\): it is therefore ample. By openess of ampleness for the family \(\overline{\Flag_{s}(T_{\X^{reg}/T})} \to T\), \(L_{t}\) remains ample for \(t\) in a neighborhood of \(t_{0}\). Since \(\{t \in T \ | \ X_{t} \ \text{is smooth} \}\) is a dense open subset of \(T\), there exists \(t \in T\) such that \(\L_{\lambda}(TX_{t})\) is ample, which concludes the proof.
\end{proof}

\section{Applications in hyperbolicity}
\label{se: hyperbolicity}
A particular case of Theorem~\ref{thm: main theorem2} for \(\lambda=(1^{k})\) implies that a generic complete intersection \(X=H_{1} \cap \dotsb \cap H_{c}\) of codimension \(c\geq \frac{N}{1+k}\), with \(H_{i} \in \O_{M}(d_{i})\) of degree \(d_{i} \geq \big(1+ 2c(k+1)(N+k(N-k))^{c(k+1)+1}\big)^{2}\),  has the exterior power of its cotangent bundle \(\Ext^{k}\Omega_{X}\) ample. We will explain in the sequel what can be deduced of \(X\) in term of \(k\)-infinitesimal hyperbolicity.

\subsection*{\texorpdfstring{\(k\)}{k}-infinitesimal hyperbolicity}
We briefly recall the notion of \(k\)-infinitesimal hyperbolicity, and we refer to~\cite{Santa} for more details on the subject. Let \(X\) be a compact complex manifold of dimension \(N\), and let \(1\leq k \leq N\). Define the so-called \textsl{\(k\)-Kobayashi-Eisenman infinitesimal pseudometric} on all decomposable \(k\)-vectors \( (x, \xi=v_{1} \wedge \dotsb \wedge v_{k}) \) of \(\Ext^{k}TX\) as follows:
\[
  \mathbf{e}^{k}(x, \xi)
  =
  \inf
  \{
    \frac{1}{R} > 0
    \ | \
    \exists f\colon \B^{k} \to X, f(0)=x,
    f_{*}(
    \partial/\partial t_{1} \wedge \dotsb \wedge \partial/\partial t_{k}
    )
    =
    R \xi
  \},
\]
where \(\B^{k}\) is the unit ball in \(\C^{k}\) and \(R\in \R_{>0}\). Consider \(\Grass(k,TX) \subset \P(\Ext^{k}TX)\), and denote \(\O(-1)\) the restriction to \(\Grass(k,TX)\) of the tautological line bundle on \(\P(\Ext^{k}TX)\). Define a function \(H^{k}: \O(-1) \to \R_{+}\) by setting for
\(x \in X\) and \(\xi=v_{1}\wedge \dotsb \wedge v_{k}\):
\[
  H^{k}\big(((x,[\xi]), \xi)\big)
  =
  e^{k}\big((x, \xi)\big).
\]
We say that \(X\) is \emph{\(k\)-infinitesimally hyperbolic} if for any smooth metric \(\omega\) on \(\O(-1)\), there exists \(\varepsilon > 0\) such that
\[
  H^{k} \geq \varepsilon \omega
\]
on \( X \).
Note that for \(k=1\), being \(1\)-infinitesimally hyperbolic is equivalent to being hyperbolic in the sense of Kobayashi, which is also equivalent to the non-existence of entire curves in \(X\) (recall that \(X\) is compact).
Similarly, observe that for \(k \geq 1\), being \(k\)-hyperbolic implies in particular that every holomorphic map \(f\colon\C\times\B^{k-1}\to X\) must be degenerate in the sense that the jacobian matrix of \(f\) is never of maximal rank. Indeed, otherwise, take \(z\) such that \(df_{z}\) is injective. Without loss of generality, suppose \(z=0\), and consider
\[
  f_{R}\colon B^{k}
  \to
  X,
  \
  (z_{1}, z_{2}, \dotsc, z_{k})
  \mapsto
  f(Rz_{1}, z_{2}, \dotsc, z_{k}).
\]
Then, denoting \(\xi=f_{*}(\partial/\partial t_{1} \wedge \dotsb \wedge \partial/\partial t_{k}) \neq 0\), one has
\[
  (f_{R})_{*}(\partial/\partial t_{1} \wedge \dotsb \wedge \partial/\partial t_{k})
  =
  R \xi.
\]
Therefore, letting \(R \to \infty\), one sees that \((f(0), \xi)\) is a degenerate point for \(\mathbf{e}^{k}\), which contradicts the \(k\)-infinitesimal hyperbolicity.

As a generalization of a theorem due to Kobayashi (\cite{kob13}) asserting that ampleness of the cotangent bundle implies hyperbolicity, we have the following (which is a particular case of~\cite[Prop. 3.4]{Santa}):
\begin{proposition}[\cite{Santa}]
  Let \(X\) be a compact complex manifold, with \(\Ext^{k}\Omega_{X}\) ample. Then \(X\) is \(k\)-infinitesimally hyperbolic.
\end{proposition}

This proposition combined with Theorem~\ref{thm: main theorem2} gives the following.
\begin{corollary}
  Let \(X=H_{1}\cap \dotsb \cap H_{c}\) be a generic complete intersection in \(\P^{N}\) of codimension \(c \geq \frac{N}{k+1}\), with \(H_{i} \in \O_{\P^{N}}(d_{i})\) of degree \(d_{i} \geq \big(1+2c(k+1)(N+k(N-k))^{c(k+1)+1}\big)^{2}\). Then \(X\) is \(k\)-infinitesimally hyperbolic.
\end{corollary}

The Kobayashi conjecture for general hypersurfaces in projective spaces was recently proved by Brotbek in~\cite{brotbek2017hyperbolicity}: a general hypersurface in \(\P^{N}\) of degree greater than a bound \(d_{N}=(N+1)^{2N+6}\) is hyperbolic. Let us mention here that, as hyperbolic implies \(k\)-infinitesimally hyperbolic for \(1\leq k \leq N\), the previous result of Brotbek implies in particular this corollary (bounds included). However, using Theorem~\ref{thm: main theorem} instead of Theorem~\ref{thm: main theorem2}, it is possible to construct particular examples that are not covered by Brotbek's theorem. For instance, a generic complete intersection \(H_{1}\cap H_{2}\) of codimension 2, with \(H_{i} \in H^{0}(\P^{N}, \O_{\P^{N}}(d))\), \(d \geq d_{N}'\), where \(d_{N}'=4(\lceil \frac{N}{2} -1 \rceil+1)\big(N+\lceil \frac{N}{2} -1 \rceil(N-\lceil \frac{N}{2} -1 \rceil)\big)^{N+1}\), is \(\big(N-\lceil \frac{N}{2} -1 \rceil\big)\)-infinitesimally hyperbolic. Note that \(d_{N}' \leq 2(N+1)(\frac{N-2}{2})^{2N+2}\) is indeed a better bound than \(d_{N}\) for large \(N\).

\bibliographystyle{alpha}
\bibliography{cotangent}

\begin{thebibliography}{{Dem}97}

\bibitem[BD18]{BD2018}
Damian {Brotbek} and Lionel {Darondeau}.
\newblock {Complete intersection varieties with ample cotangent bundles}.
\newblock {\em {Invent. Math.}}, 212(3):913--940, 2018.

\bibitem[{Ben}11]{Benoit}
Olivier {Benoist}.
\newblock {Le th\'eor\`eme de Bertini en famille}.
\newblock {\em {Bull. Soc. Math. Fr.}}, 139(4):555--569, 2011.

\bibitem[{Bot}57]{Bott}
Raoul {Bott}.
\newblock {Homogeneous vector bundles}.
\newblock {\em {Ann. Math. (2)}}, 66:203--248, 1957.

\bibitem[BR90]{BR}
P~Br{\"u}ckmann and H-G Rackwitz.
\newblock T-symmetrical tensor forms on complete intersections.
\newblock {\em Mathematische Annalen}, 288(1):627--635, 1990.

\bibitem[Bro16]{Brotbek2016}
Damian Brotbek.
\newblock Symmetric differential forms on complete intersection varieties and
  applications.
\newblock {\em Mathematische Annalen}, 366(1-2):417--446, 2016.

\bibitem[Bro17]{brotbek2017hyperbolicity}
Damian Brotbek.
\newblock On the hyperbolicity of general hypersurfaces.
\newblock {\em Publications math{\'e}matiques de l'IH{\'E}S}, 126(1):1--34,
  2017.

\bibitem[Deb05]{Deb2005}
Olivier Debarre.
\newblock Varieties with ample cotangent bundle.
\newblock {\em Compositio Mathematica}, 141(6):1445--1459, 2005.

\bibitem[Dem88]{Demailly}
Jean-Pierre Demailly.
\newblock Vanishing theorems for tensor powers of an ample vector bundle.
\newblock {\em Inventiones mathematicae}, 91(1):203--220, 1988.

\bibitem[{Dem}97]{Santa}
Jean-Pierre {Demailly}.
\newblock {Algebraic criteria for Kobayashi hyperbolic projective varieties and
  jet differentials}.
\newblock In {\em {Algebraic geometry. Proceedings of the Summer Research
  Institute, Santa Cruz, CA, USA, July 9--29, 1995}}, pages 285--360.
  Providence, RI: American Mathematical Society, 1997.

\bibitem[{Den}20]{deng2017diverio}
Ya~{Deng}.
\newblock {On the Diverio-Trapani conjecture}.
\newblock {\em {Ann. Sci. \'Ec. Norm. Sup\'er. (4)}}, 53(3):787--814, 2020.

\bibitem[Kob13]{kob13}
Shoshichi Kobayashi.
\newblock {\em Hyperbolic complex spaces}, volume 318.
\newblock Springer Science \& Business Media, 2013.

\bibitem[Laz04]{Laz}
Robert Lazarsfeld.
\newblock {\em Positivity in algebraic geometry. A Series of Modern Surveys in
  Mathematics}.
\newblock Springer Berlin, 2004.

\bibitem[LN19]{LN2018}
F.~{Laytimi} and W.~{Nahm}.
\newblock {Ampleness equivalence and dominance for vector bundles}.
\newblock {\em {Geom. Dedicata}}, 200:77--84, 2019.

\bibitem[Wey03]{Wey}
Jerzy Weyman.
\newblock {\em Cohomology of vector bundles and syzygies}, volume 149.
\newblock Cambridge University Press, 2003.

\bibitem[Xie18]{Xie2018}
Song-Yan Xie.
\newblock On the ampleness of the cotangent bundles of complete intersections.
\newblock {\em Inventiones mathematicae}, 212(3):941--996, 2018.

\end{thebibliography}

\end{document}